\documentclass[preprint]{imsart}
\pdfoutput=1
\RequirePackage[OT1]{fontenc}
\usepackage{amsthm,amsmath,mathtools,natbib,amssymb,bm,enumitem}
\RequirePackage[colorlinks,citecolor=blue,urlcolor=blue]{hyperref}
\usepackage{booktabs}
\usepackage{tikz}

\startlocaldefs
\numberwithin{equation}{section}
\theoremstyle{plain}
\newtheorem{thm}{Theorem}[section]
\newtheorem{lemma}{Lemma}[section]
\newtheorem{corollary}{Corollary}[section]
\newtheorem{definition}{Definition}[section]
\theoremstyle{remark}
\newtheorem{remark}{Remark}[section]

\endlocaldefs
\def\citeapos#1{\citeauthor{#1}'s (\citeyear{#1})}
\newcommand{\vxx}{v_x }
\newcommand{\vyy}{v_y }
\newcommand{\U}{{ \mathrm{\scriptscriptstyle U} }}
\newcommand{\HH}{{ \mathrm{\scriptscriptstyle H} }}
\newcommand{\KL}{{ \mathrm{\scriptscriptstyle KL} }}
\newcommand{\bD}{{ \;||\; }}
\newcommand{\mymid}{\!\mid\!}
\newcommand{\mbd}{_{\mathbb{R}^d}}
\newcommand{\T}{{ \mathrm{\scriptscriptstyle T} }}
\newcommand{\rd}{\mathrm{d}}
\allowdisplaybreaks

\begin{document}

\begin{frontmatter}
\title{Harmonic Bayesian prediction under $\alpha$-divergence}
\runtitle{A Bayesian prediction under $\alpha$-divergence}

\begin{aug}
\author{\fnms{Yuzo} \snm{Maruyama}\thanksref{m1}
 \ead[label=e1]{maruyama@csis.u-tokyo.ac.jp}}
 \and
\author{\fnms{Toshio} \snm{Ohnishi}\thanksref{m2}
\ead[label=e2]{ohnishi@econ.kyushu-u.ac.jp}}


\runauthor{Y. Maruyama and T. Ohnishi}

\affiliation{The University of Tokyo and Kyushu University}
\address{University of Tokyo\thanksmark{m1} and Kyushu University\thanksmark{m2} \\
\printead{e1,e2}}
\end{aug}
\begin{abstract}
We investigate Bayesian shrinkage methods for constructing predictive distributions.
We consider the multivariate normal model with a known covariance matrix
and show that the Bayesian predictive density with respect to Stein's harmonic prior
dominates the best invariant Bayesian predictive density, when the dimension is greater than three.
Alpha-divergence from the true distribution to a predictive distribution is adopted as a loss function.
\end{abstract}


\end{frontmatter}

\section{Introduction}
\label{sec:intro}
Let $X\sim N_d(\mu,\vxx I)$ and $Y\sim N_d(\mu,\vyy I)$
be independent $d$-dimensional multivariate normal vectors with common unknown mean $\mu$.
We assume that $d\geq 3$ and that $\vxx $ and $\vyy $ are known.
Let $\phi(\cdot,\sigma^2)$ be the probability density of $N_d(0,\sigma^2I)$.
Then the probability density of $X$ and that of $Y$ are $\phi(x-\mu,\vxx )$ and $\phi(y-\mu,\vyy )$,
respectively.

Based on only observing $X=x$, we consider the problem of obtaining a predictive density
$\hat{p}(y\mymid x)$ for $Y$ that is close to the true density $\phi(y-\mu,\vyy )$.
In most earlier papers on such prediction problems, 
a predictive density $\hat{p}(y\mymid x)$ 
is often evaluated by 
\begin{equation}\label{KL-loss}
 D_\KL\left\{\phi(y-\mu,\vyy )\bD\hat{p}(y\mymid x)\right\}
  =\int\mbd \phi(y-\mu,\vyy ) \log
  \frac{\phi(y-\mu,\vyy )}{\hat{p}(y\mymid x)} \rd y,
\end{equation}
which is called the Kullback-Leibler divergence loss (KL-div loss) from 
$\phi(y-\mu,\vyy )$ to $\hat{p}(y\mymid x)$.
The overall quality of the procedure $\hat{p}(y\mymid x)$
for each $\mu$ is then
summarized by the Kullback-Leibler divergence risk 
\begin{equation}\label{KL-risk}
 R_\KL\{\phi(y-\mu,\vyy )\bD \hat{p}(y\mymid x)\}
  =\int\mbd D_\KL\left\{\phi(y-\mu,\vyy )\bD \hat{p}(y\mymid x)\right\}\phi(x-\mu,\vxx )\rd x.
\end{equation}

\cite{Aitchison-1975} showed that 
the Bayesian solution with respect to
a prior $\pi(\mu)$ under KL-div loss given by \eqref{KL-loss} is 
the Bayesian predictive density
\begin{equation}\label{Bayes-predictive}
  \hat{p}_{\pi}(y\mymid x) 
    =\int\mbd \phi(y-\mu,\vyy ) \pi(\mu\mymid x)\rd\mu  ,
\end{equation}
where 
$ \pi(\mu \mymid x)=\phi(x-\mu,\vxx )\pi(\mu)/m_\pi(x,\vxx )$
is the posterior density corresponding to $\pi(\mu)$ and 
\begin{equation}\label{eq:marginal}
 m_\pi(x,v)=\int\mbd \phi(x-\mu,v)\pi(\mu)\rd\mu 
\end{equation}
is the marginal density of $X \sim N_d(\mu,vI)$ under the prior $\pi(\mu)$.

For the prediction problems in general, 
many studies suggest the use of the Bayesian predictive
density rather than plug-in densities of the form
\begin{equation*}
 \phi(y-\hat{\mu}(x),\vyy),
\end{equation*}
where $\hat{\mu}(x)$ is an estimated value of $\mu$. 
\cite{Liang-Barron-2004} showed that the Bayesian predictive
density with respect to the uniform prior
\begin{equation}\label{right-invariant-prior-known}
 \pi_\U(\mu)=1,
\end{equation}
which is given by
\begin{equation}\label{eq:best_equivariant}
 \hat{p}_\U(y\mymid x)
  =\int\mbd \phi(y-\mu,\vyy)\pi_\U(\mu\mymid x)\rd\mu =\phi(y-x,\vxx+\vyy)
\end{equation}
is best invariant and minimax.
Although the best invariant Bayesian predictive density is generally a good default procedure,
it has been shown to be inadmissible in some cases.
Specifically, 
\cite{Komaki-2001} showed that 
the Bayesian predictive density with respect to \citeapos{Stein-1974} harmonic prior 
\begin{equation}
\pi_\HH(\mu)=\|\mu\|^{-(d-2)}
\end{equation}
dominates the best invariant Bayesian predictive density  $ \hat{p}_\U(y\mymid x)$.
\cite{George-etal-2006} extended \citeapos{Komaki-2001} result to
general shrinkage priors including \citeapos{Strawderman-1971}  prior.

From a more general viewpoint,
the KL-div loss given by \eqref{KL-loss}
is in the class of $\alpha$-divergence loss ($\alpha$-div loss)
introduced by \cite{Csiszar-1967} and defined by
\begin{equation} \label{def:alpha.div}
 D_{\alpha}\left\{\phi(y-\mu,\vyy )\bD\hat{p}(y\mymid x)\right\}
  =\int\mbd f_{\alpha}\left(\frac{\hat{p}(y\mymid x)}{\phi(y-\mu,\vyy )}\right)\phi(y-\mu,\vyy )\rd y ,
\end{equation}
where
\begin{equation*}
 f_{\alpha}(z)=
  \begin{cases}
   \left\{4/(1-\alpha^2)\right\}\left\{1- z^{(1+\alpha)/2}\right\}, & \ |\alpha| <1, \\
   z\log z, & \ \alpha =1, \\
   -\log z, & \ \alpha = -1.
  \end{cases}
\end{equation*}
When $\alpha=-1$, we have
\begin{align*}
 D_{-1}\left\{\phi(y-\mu,\vyy )\bD\hat{p}(y\mymid x)\right\}
 = D_\KL\left\{\phi(y-\mu,\vyy )\bD\hat{p}(y\mymid x)\right\},
\end{align*}
where $D_\KL$ is given by \eqref{KL-loss}.
When $\alpha=0$, we have
\begin{align*}
 f_0(z)&=4(1-z^{1/2}), \\ 
 D_0\left\{\phi(y-\mu,\vyy )\bD\hat{p}(y\mymid x)\right\}
 &=2\int\mbd \left\{\hat{p}^{1/2}(y\mymid x)-\phi^{1/2}(y-\mu,\vyy )\right\}^2 \rd y,
\end{align*}
where $ \sqrt{D_0\left\{\phi(y-\mu,\vyy )\bD\hat{p}(y\mymid x)\right\}}/2$
is the Hellinger distance between $\hat{p}(y\mymid x)$ and $\phi(y-\mu,\vyy )$.
As in the Kullback-Leibler divergence risk given by \eqref{KL-risk}, the overall quality of the procedure $\hat{p}(y\mymid x)$
for each $\mu$ is 
summarized by the $\alpha$-divergence risk
\begin{equation*}
 R_\alpha\{\phi(y-\mu,\vyy )\bD \hat{p}(y\mymid x) \}
  = \int\mbd D_\alpha\left\{\phi(y-\mu,\vyy )\bD \hat{p}(y\mymid x) \right\}\phi(x-\mu,\vxx )\rd x.
\end{equation*}
\cite{Corcuera-Giummole-1999} showed that a Bayesian predictive density under $\alpha$-div loss
is
\begin{equation} \label{gene-B-density}
 \hat{p}_{\pi}(y\mymid x; \alpha) \propto
\begin{cases}
 \displaystyle
 \left\{\int\mbd \phi^{\frac{1-\alpha}{2}}(y-\mu,\vyy )\phi(x-\mu,\vxx )\pi(\mu)\rd\mu  \right\}^{\frac{2}{1-\alpha}},
 & -1\leq \alpha < 1, \\
 \displaystyle
 \exp\left( \int\mbd \left\{\log \phi(y-\mu,\vyy )\right\}\phi(x-\mu,\vxx )\pi(\mu)\rd\mu \right),
 & \alpha=1.
\end{cases}
\end{equation}
By \eqref{gene-B-density}, in the prediction problem under $\alpha$-div loss with $\alpha=1$
from the Bayesian point of view, the Bayesian solution is the normal
density
\begin{align}
 \hat{p}_{\pi}(y\mymid x; 1)=\phi(y-\hat{\mu}_\pi(x),\vyy ),
\end{align}
where $\hat{\mu}_\pi(x)$ is the posterior mean given by
\begin{equation}
 \hat{\mu}_\pi(x)
  =\int\mbd \mu\pi(\mu\mymid x)\rd\mu =x+\vxx\nabla_x \log m(x,\vxx). 
\end{equation}
In general, the Bayesian prediction problem under $\alpha = 1$ reduces to the
estimation problem under the KL-div loss in the case of the exponential family density.
This is because the exponential family density is closed under the calculation
in \eqref{gene-B-density} with $\alpha = 1$, as pointed out in \cite{Yanagimoto-Ohnishi-2009}.

As demonstrated in \cite{Maruyama-Strawderman-2012}, the $\alpha$-div loss in the case of $\alpha=1$
is written as
\begin{equation*}
 D_1\left\{\phi(y-\mu,\vyy )\bD \phi(y-\hat{\mu}_\pi(x),\vyy )\right\}
  =\frac{\|\hat{\mu}_\pi(x)-\mu\|^2}{2\vyy },
\end{equation*}
and hence the prediction problem under $\alpha=1$ reduces to the estimation problem of $\mu$
under the quadratic loss.
\cite{Stein-1981} showed that
\begin{align}
 E_X\left[\|\hat{\mu}_\pi(X)-\mu\|^2\right]
 =d\vxx +4\vxx^2 E_X\left[\frac{\Delta_x m_\pi^{1/2}(X,\vxx)}{m_\pi^{1/2}(X,\vxx)}\right],
\end{align}
which implies that the risk difference under $\alpha=1$ is expressed as
\begin{equation} \label{eq:beautiful_1}
 \begin{split}
  & R_1\{\phi(y-\mu,\vyy)\bD \hat{p}_\U(y\mymid x; 1) \}-
  R_1\{\phi(y-\mu,\vyy)\bD \hat{p}_\pi(y\mymid x; 1) \}  \\
  &=\frac{2\vxx^2}{\vyy}
  E_X\left[-\frac{\Delta_x m^{1/2}_\pi(X,\vxx)}{m^{1/2}_\pi(X,\vxx)}\right].
 \end{split}
\end{equation}

Under the KL-div loss or $\alpha$-div loss with $\alpha=-1$, 
\cite{George-etal-2006} showed that the risk difference is given by
\begin{equation}\label{eq:beautiful}
 \begin{split}
  & R_{-1}\{\phi(y-\mu,\vyy )\bD \hat{p}_{\U}(y\mymid x;-1) \}-
  R_{-1}\{\phi(y-\mu,\vyy )\bD \hat{p}_\pi(y\mymid x;-1) \}  \\
  &=2\int_{v_*}^{\vxx }E_Z\left[-\frac{\Delta_z m^{1/2}_\pi(Z,v)}{m^{1/2}_\pi(Z,v)}\right]\rd v, 
\end{split}
\end{equation}
where $\hat{p}_{\U}(y\mymid x;-1)$ is given by \eqref{eq:best_equivariant}, $Z\sim N_d(\mu,vI)$
and $v_*=\vxx \vyy /(\vxx +\vyy )$.
From this viewpoint, 
\cite{George-etal-2006} and \cite{Brown-etal-2008} considered the prediction problem 
under $\alpha$-div loss with two extreme cases $\alpha=\pm 1 $
and found a beautiful relationship of risk differences for two cases
via $\Delta_z \{m_\pi(z,v)\}^{1/2}$ for some $v$.
Under both risks $R_1$ and $R_{-1}$, any shrinkage prior of the satisfier
of the superharmonicity 
\begin{equation}\label{super}
 \Delta_z m^{1/2}_\pi(z,v) \leq 0 \  \text{ for }
  \begin{cases}
   \forall v \in (v_*,\vxx) \text{ for }\alpha=-1, \\
   v=\vxx \text{ for }\alpha=1,
  \end{cases}
\end{equation}
implies the improvement over the best invariant Bayesian procedure.
It is well-known that the superharmonicity of $\pi(\mu)$, $ \Delta_\mu \pi(\mu)\leq 0$, 
implies the superharmonicity of $m_\pi(z,v)$, $ \Delta_z m_\pi(z,v)\leq 0$.
Further the superharmonicity of $m_\pi(z,v)$ implies the superharmonicity of $\{m_\pi(z,v)\}^{1/2}$.
Hence the harmonic prior $\pi_\HH(\mu)=\|\mu\|^{-(d-2)}$ is one of 
the satisfiers of the superharmonicity of $\{m_\pi(z,v)\}^{1/2}$.

Because of the relationship given by \eqref{eq:beautiful_1},
\eqref{eq:beautiful} and \eqref{super}, it is of great interest to find 
the corresponding link via $ \Delta_z\{m_\pi(z,v)\}^{1/2}$
for $\alpha$-div loss with general $\alpha\in(-1,1)$.
To our knowledge, decision-theoretic properties seem to depend on the general
structure of the problem (the general type of problem (location, scale), and
the dimension of the parameter space) and on the prior in a Bayesian-setup,
but not on the loss function, as \cite{Brown-1979} pointed out in the estimation problem.

In this paper, we investigate 
the risk difference, $ \mathrm{diff}R_{\alpha,\U,\pi}$, in the case of  $\alpha$-div loss, 
defined by
\begin{equation}\label{diff.alpha}
 \begin{split}
  \mathrm{diff}R_{\alpha,\U,\pi} 
  & =R_\alpha\left\{\phi(y-\mu,\vyy )\bD \hat{p}_{\U}(y\mymid x ; \alpha) \right\} \\ & \quad -
  R_\alpha\left\{\phi(y-\mu,\vyy )\bD \hat{p}_{\pi}(y\mymid x ; \alpha) \right\}.
 \end{split}
\end{equation}
In \eqref{diff.alpha}, $\hat{p}_{\pi}(y\mymid x ; \alpha)$ is given by \eqref{gene-B-density}
and $ \hat{p}_{\U}(y\mymid x ; \alpha)$ is
the Bayesian predictive density under the uniform prior \eqref{right-invariant-prior-known}, the form of which will be derived in
\eqref{bestequiv} of Section \ref{sec:BPD}.
As a generalization of \citeapos{Liang-Barron-2004} result,
$ \hat{p}_{\U}(y\mymid x ; \alpha)$ for general $\alpha\in(-1,1)$ is best invariant and minimax, as shown in Appendix \ref{sec:minimaxity}.
Further, analyzing $\mathrm{diff}R_{\alpha,\U,\pi} $,
we provide some asymptotic results and a non-asymptotic decision-theoretic result.
\begin{description}[leftmargin=7pt]
 \item[Asymptotic results] 
We show not only somewhat expected relationship
\begin{equation}\label{eq:limit_1}
 \lim_{\alpha \to 1-0}\mathrm{diff}R_{\alpha,\U,\pi}
  =\mathrm{diff}R_{1,\U,\pi}, \
  \lim_{\alpha \to -1+0}\mathrm{diff}R_{\alpha,\U,\pi}
  =\mathrm{diff}R_{-1,\U,\pi},
\end{equation}
where $\mathrm{diff}R_{1,\U,\pi}$ and $\mathrm{diff}R_{-1,\U,\pi}$ are given in
\eqref{eq:beautiful_1} and \eqref{eq:beautiful} respectively, but also
the asymptotic relationship for general $\alpha\in(-1,1)$,
\begin{equation}\label{eq:limit_2}
 \lim_{\vxx/\vyy \to +0}\mathrm{diff}R_{\alpha,\U,\pi}
  =\mathrm{diff}R_{1,\U,\pi}.
\end{equation}
Hence, the asymptotic situation $ \vxx/\vyy \to 0$ corresponds to the case $\alpha\to 1$ and 
$ \Delta_z\{m_\pi(z,v)\}^{1/2}$ plays an important role for general $\alpha\in(-1,1)$.
\item[Non-asymptotic result] 
We particularly investigate a decision-theoretic property of
the Bayesian predictive density with respect to $\pi_\HH(\mu)=\|\mu\|^{-(d-2)}$
under $\alpha$-div loss with general $\alpha\in(-1,1)$.
We show that, 
the Bayesian predictive density with respect to $\pi_\HH(\mu)=\|\mu\|^{-(d-2)}$ dominates
the best invariant Bayesian predictive density with respect to $\pi_\U(\mu)=1$
if
\begin{equation*}
 \frac{\vxx}{\vyy} \leq
  \begin{cases}
   \displaystyle
   \frac{d+2}{d(1+\alpha)}
   & \displaystyle\text{if }\frac{2}{1-\alpha}\text{ is a positive integer}, \\
   \displaystyle
   \left(\frac{2}{1-\alpha}\right)^2\frac{d+2}{d}\frac{1-\left\{\kappa-2/(1-\alpha)\right\}}{2\kappa(\kappa-1)}
   &\text{ otherwise},
  \end{cases}
\end{equation*}
where $\kappa$ is the smallest integer larger than $2/(1-\alpha)$.
\end{description}
The organization of this paper is as follows. In Section \ref{sec:BPD},
we derive the exact form of $\hat{p}_{\pi}(y\mymid x ; \alpha)$,
propose a general sufficient condition for
$\mathrm{diff}R_{\alpha,\U,\pi}\geq 0$, where $\mathrm{diff}R_{\alpha,\U,\pi}$ is given by
\eqref{diff.alpha},
and demonstrate the asymptotic relationship described in \eqref{eq:limit_1} and \eqref{eq:limit_2}.
In Section \ref{sec:harmonic}, we propose the non-asymptotic result under the harmonic prior
$\pi_\HH(\mu)=\|\mu\|^{-(d-2)}$ described in the above. 
Some technical proofs are given in Sections \ref{sec:minimaxity} and \ref{sec:proof.harmonic} of Appendix.

\section{Bayesian predictive density under $\alpha$-divergence loss}
\label{sec:BPD}
As in \eqref{gene-B-density}, the Bayes predictive density under $\alpha$-div loss is
\begin{align}\label{eq:general_Bayes}
 \hat{p}_\pi(y\mymid x ; \alpha) \propto \left\{\int\mbd
 \phi\left(x-\mu,\vxx \right)\phi^\beta( y -\mu,\vyy ) \pi(\mu)\rd\mu \right\}^{1/\beta},
\end{align}
where
\begin{equation}\label{def:beta}
 \beta=\frac{1-\alpha}{2}.
\end{equation}
Clearly, it follows from $\alpha\in(-1,1)$ that $0<\beta<1$.
Let
\begin{equation}\label{gammagamma}
 \gamma=\frac{1}{1+\beta\vxx /\vyy}.
\end{equation}
Since the relation of completing squares with respect to $\mu$, for  
$\phi\left(x-\mu,\vxx \right)\phi^\beta( y -\mu,\vyy )$,
is given by
\begin{align*}
 & \frac{1}{\vxx }\|x-\mu\|^2+\frac{\beta}{\vyy }\|y-\mu\|^2 \\
 &=\frac{1}{\vxx }\left(\|x-\mu\|^2+\frac{1-\gamma}{\gamma}\|y-\mu\|^2\right) \\
 &=\frac{1}{\vxx }\left(\frac{1}{\gamma}\|\mu-\{\gamma x+(1-\gamma)y\}\|^2- \frac{\|\gamma x+(1-\gamma)y\|^2}{\gamma}+\|x\|^2+\frac{1-\gamma}{\gamma}\|y\|^2\right) \\
 &=\frac{1}{\vxx }\left\{\frac{1}{\gamma}\|\mu-\{\gamma x+(1-\gamma)y\}\|^2+
 (1-\gamma)\|y- x\|^2 \right\}\\
 & =\frac{1}{\vxx \gamma}\|\mu-\{\gamma x+(1-\gamma)y\}\|^2+\beta\frac{\gamma}{\vyy}\|y-x\|^2,
\end{align*}
we have the identity,
\begin{equation}\label{eq:hirei}
 \begin{split}
  &\phi\left(x-\mu,\vxx \right)\phi^\beta( y -\mu,\vyy ) \\      
  &=\gamma^{(1-\beta)d/2}\phi(\gamma x+(1-\gamma)y-\mu,\vxx \gamma)\phi^\beta(y-x,\vyy /\gamma). 
 \end{split}
\end{equation}
Under the uniform prior $\pi_\U(\mu)=1$, we have, from \eqref{eq:hirei},
\begin{align*}
 \int\mbd \phi\left(x-\mu,\vxx \right)\phi^\beta( y  -\mu,\vyy )\pi_\U(\mu)\rd\mu 
 =\gamma^{(1-\beta)d/2}\phi^\beta(y-x,\vyy /\gamma)
\end{align*}
in \eqref{eq:general_Bayes}. Therefore the Bayesian predictive density under the uniform prior is
\begin{equation}\label{bestequiv}
 \hat{p}_{\U}(y\mymid x ; \alpha)
  =\phi(y-x,\vyy /\gamma)=\phi(y-x,\vyy+\beta\vxx),
\end{equation}
which is the target predictive density so that the risk difference
\begin{equation*}
 \mathrm{diff}R_{\alpha,\U,\pi}
  = R_\alpha\left\{\phi(y-\mu,\vyy )\bD \hat{p}_{\U}(y\mymid x ; \alpha) \right\}-
  R_\alpha\left\{\phi(y-\mu,\vyy )\bD \hat{p}_{\pi}(y\mymid x ; \alpha) \right\}
\end{equation*}
is going to be investigated in this paper.
As shown in Appendix \ref{sec:minimaxity}, 
$ \hat{p}_{\U}(y\mymid x ; \alpha)$ for general $\alpha\in(-1,1)$ is best invariant and minimax,
which is regarded as a generalization of \citeapos{Liang-Barron-2004} minimaxity result.
Hence $ \hat{p}_{\pi}(y\mymid x ; \alpha)$
with $\mathrm{diff}R_{\alpha,\U,\pi}\geq 0$ for all $\mu\in\mathbb{R}^d$ is minimax.

The exact form of Bayes predictive density $\hat{p}_\pi(y\mymid x; \alpha)$ for \eqref{eq:general_Bayes}
with normalizing constant, which is regarded as a generalization of Theorem 1 of \cite{Komaki-2001}
as well as
Lemma 2 of \cite{George-etal-2006}, is provided as follows.
\begin{thm}\label{thm:Bayes.predictive.density}
The Bayes predictive density under $\pi(\mu)$ is
 \begin{equation}\label{form:Bayes.predictive.density}
  \hat{p}_\pi(y\mymid x; \alpha)
   =\frac{m^{1/\beta}_\pi(\gamma x+(1-\gamma)y,\vxx\gamma)}
   {E_{Z_1}\left[m^{1/\beta}_\pi(x+\xi Z_1,\vxx\gamma)\right]}\hat{p}_{\U}(y\mymid x;\alpha),
 \end{equation}
 where $Z_1\sim N_d(0,I)$ and
 \begin{equation}\label{eq:xixixi}
  \xi = (1-\gamma)(\vyy /\gamma)^{1/2}.
 \end{equation}
\end{thm}
\begin{proof}
By \eqref{eq:general_Bayes}, \eqref{eq:hirei} and \eqref{bestequiv}, we have
 \begin{equation}\label{eq:hatpppp}
  \hat{p}_\pi(y\mymid x; \alpha) 
   \propto \phi(y-x,\vyy /\gamma)m^{1/\beta}_\pi(\gamma x+(1-\gamma)y,\vxx\gamma).
 \end{equation}
The normalizing constant of \eqref{eq:hatpppp} is 
 \begin{align*}
  &\int\mbd \phi(y-x,\vyy /\gamma)m^{1/\beta}_\pi(\gamma x+(1-\gamma)y,\vxx\gamma) \rd y   \\
  &=\int\mbd \phi(z_1,1) m^{1/\beta}_\pi\left(x+(1-\gamma)(\vyy /\gamma)^{1/2}z_1,\vxx\gamma\right) \rd z_1 \\
  &=E_{Z_1}\left[m^{1/\beta}_\pi(x+\xi Z_1,\vxx\gamma)\right],
 \end{align*}
where the first equality is from the transformation, $z_1=(\gamma/\vyy )^{1/2}(y-x)$.
\end{proof}

In the following, as a generalization of the Bayes predictive density, 
we consider
 \begin{equation}\label{form:Bayes.predictive.density.f}
  \hat{p}_f(y\mymid x; \alpha)
   =\frac{f(\gamma x+(1-\gamma)y)}
   {E_{Z_1}\left[f(x+\xi Z_1)\right]}\hat{p}_{\U}(y\mymid x;\alpha)
 \end{equation}
 where $f:\mathbb{R}^d\to\mathbb{R}_+$ is general.
 As in the proof of Theorem \ref{thm:Bayes.predictive.density},
 $\int\hat{p}_f(y\mymid x; \alpha)\rd y =1$ follows. Also
 $\hat{p}_f(y\mymid x; \alpha)$ is nonnegative for any $y\in\mathbb{R}^d$ and hence
 $\hat{p}_f(y\mymid x; \alpha)$ is regarded as a predictive density.
 
By the definition of the $\alpha$-div loss given by \eqref{def:alpha.div}, the risk difference
between $\hat{p}_{\U}$ and $\hat{p}_f$ is written as
\begin{equation}\label{risk.diff}
 \begin{split}
  &\mathrm{diff}R_{\alpha,\U,f} \\
  & =R_\alpha\{\phi(y-\mu,\vyy )\bD \hat{p}_{\U}(y\mymid x; \alpha) \}-
  R_\alpha\{\phi(y-\mu,\vyy ) \bD \hat{p}_f(y\mymid x; \alpha)\} \\
  &=\frac{1}{\beta(1-\beta)}
  \int_{\mathbb{R}^{2d}}\left\{\left(\frac{\hat{p}_f(y\mymid x;\alpha)}{\phi(y-\mu,\vyy )}\right)^{1-\beta}
  -\left(\frac{\hat{p}_{\U}(y\mymid x;\alpha)}{\phi(y-\mu,\vyy )}\right)^{1-\beta}
  \right\}  \\
  & \qquad  \times \phi(x-\mu,\vxx )\phi(y-\mu,\vyy ) \rd x \rd y .
 \end{split}      
\end{equation}
Then we have a following result.
\begin{thm}\label{thm:sufficient.1}
\begin{enumerate}
 \item\label{1.thm:sufficient.1}
       The risk difference $\mathrm{diff}R_{\alpha,\U,f}$ given by \eqref{risk.diff}
is written by $E[\rho(W,Z)]$ where $W\sim N_d(\mu,\vxx\gamma)$, $Z\sim N_d(0,I)$ and
\begin{equation}\label{eq:rho}
  \rho(w,z)=\frac{4\gamma^{(1-\beta)d/2}}{\beta^2f^{\beta-1}(w)} 
  \int_0^{\xi}t\frac{-\Delta_w\varrho(w+tz;t;f)}{\varrho^{2/\beta-1}(w+tz;t;f)}\rd t
\end{equation} 
where 
\begin{equation}\label{eq:varrho}
 \varrho(u;t;f)
  =\left\{
    E_{Z_1} \left[f(tZ_1+u)\right]
   \right\}^{\beta/2}, \text{ for }Z_1\sim N_d(0,I).
 \end{equation}
 \item\label{2.thm:sufficient.1}
 A sufficient condition for $\mathrm{diff}R_{\alpha,\U,f}\geq 0$
for $\forall \mu\in\mathbb{R}^d$ is 
\begin{align}\label{eq:suffsuff}
 \Delta_u \varrho(u;t;f)
 \leq 0\quad\forall u\in\mathbb{R}^d, \ 0\leq \forall  t \leq \xi.
\end{align}
\end{enumerate}
\end{thm}

\begin{proof}
 Part \ref{2.thm:sufficient.1} easily follows from Part \ref{1.thm:sufficient.1}
 and, in the following, we show  Part \ref{1.thm:sufficient.1}.
 
 By \eqref{eq:hirei}, \eqref{bestequiv}, and \eqref{form:Bayes.predictive.density.f},
the integrand of \eqref{risk.diff} is rewritten as
\begin{align*}
 & \left\{\left(\frac{\phi(y-\mu,\vyy )}{\hat{p}_f(y\mymid x; \alpha)}\right)^{\beta-1}
 -\left(\frac{\phi(y-\mu,\vyy )}{\hat{p}_{\U}(y\mymid x; \alpha)}\right)^{\beta-1}\right\}
 \phi(y-\mu,\vyy )\phi(x-\mu,\vxx ) \\
 &=\gamma^{(1-\beta)d/2}\left\{\left(\frac{E_{Z_1}\left[f(x+\xi Z_1)\right]}
 {f(\gamma x+(1-\gamma)y)}\right)^{\beta-1}-1\right\} \\
 &\qquad \times\phi(\gamma x+(1-\gamma)y-\mu,\vxx \gamma) \phi(y-x,\vyy /\gamma) .
\end{align*}
By the change of variables, $w=\gamma x+(1-\gamma)y$ and $z=-(\gamma/\vyy )^{1/2}(y-x)$, 
 where Jacobian of the matrix below is $(\gamma/\vyy )^{d/2}$,
\begin{equation}
 \begin{pmatrix}
  w \\ z 
 \end{pmatrix}
=\begin{pmatrix}
  \gamma I_d & (1-\gamma)I_d \\  (\gamma/\vyy )^{1/2}I_d & -(\gamma/\vyy )^{1/2}I_d
 \end{pmatrix}
 \begin{pmatrix}
  x \\ y
 \end{pmatrix},
\end{equation}
 the risk difference is expressed as
\begin{equation}\label{IERD}
 \begin{split}
  & \frac{\gamma^{(1-\beta)d/2}}{\beta(1-\beta)}
  E_{W,Z}\left[\left(E_{Z_1}\left[\frac{f(W+\xi(Z_1+Z))}{f(W)}\right] \right)^{\beta-1} -1\right]  \\
  &=\frac{\gamma^{(1-\beta)d/2}}{\beta(1-\beta)}E_W\left[ f(W)^{1-\beta}
 \left\{g(\xi;W)-g(0;W)\right\}\right]  \\
  &=\frac{\gamma^{(1-\beta)d/2}}{\beta(1-\beta)}E_W
  \left[ f(W)^{1-\beta}
 \int_0^{\xi}\frac{\partial}{\partial t}g(t;W)\rd t\right], 
 \end{split}
\end{equation}
where $\xi=(1-\gamma)(\vyy /\gamma)^{1/2}$ as in \eqref{eq:xixixi},
$W\sim N_d(\mu,\vxx \gamma I)$, $Z_1\sim N_d(0,I)$, $Z\sim N_d(0,I)$ and
\begin{align}\label{eq:ggg}
 g(t;w)
 =E_{Z}\left[E_{Z_1}\left[f(w +t\{Z_1+Z\}) \right]^{\beta-1}\right] .
\end{align}
 In the following, $ E_{Z_1}\left[f\right]=E_{Z_1}\left[f(w +t\{Z_1+z\}) \right]$
 for notational simplicity.
 Then we have
 \begin{equation}\label{eq:ggg.1}
 \begin{split}
 \frac{\partial}{\partial t}g(t;w)
  &=
  E_{Z}\left[\frac{\partial}{\partial t}\left\{E_{Z_1}\left[f \right]\right\}^{\beta-1}\right] \\
&=  (\beta-1)E_{Z}\left[\left\{E_{Z_1}\left[f \right]\right\}^{\beta-2} 
 E_{Z_1}\left[(Z_1+Z)^\T \nabla_w f \right]\right] \\
&=  (\beta-1)E_{Z}\left[\left\{E_{Z_1}\left[f \right]\right\}^{\beta-2} 
 \left(E_{Z_1}\left[Z_1^\T \nabla_w f \right]+Z^\T E_{Z_1}\left[\nabla_w f \right]\right)\right] .
\end{split}
 \end{equation}
In \eqref{eq:ggg.1}, we have
 \begin{equation}\label{eq:ggg.2}
 \begin{split}
E_{Z_1}\left[Z_1^\T \nabla_w f \right]
&=E_{Z_1}\left[Z_1^\T \frac{1}{t}\nabla_{z_1} f \right] 
 =\frac{1}{t}E_{Z_1}\left[ \Delta_{z_1} f \right] \\
 &=tE_{Z_1}\left[ \Delta_{w} f \right] =t\Delta_{w}E_{Z_1}\left[  f \right] 
 \end{split}
 \end{equation}
 where the second equality follows from the Gauss divergence theorem. Similarly we have
 \begin{equation}\label{eq:ggg.3}
  \begin{split}
&(\beta-1)E_{Z}\left[\left\{E_{Z_1}\left[f\right]\right\}^{\beta-2} 
Z^\T E_{Z_1}\left[\nabla_w f \right]\right] \\
&=(\beta-1)E_{Z}\left[\left\{E_{Z_1}\left[f\right]\right\}^{\beta-2} 
Z^\T\frac{1}{t}E_{Z_1}\left[ \nabla_{z}f \right]\right] \\
&=\frac{1}{t}(\beta-1)E_{Z}\left[\left\{E_{Z_1}\left[f\right]\right\}^{\beta-2} 
Z^\T\nabla_{z}E_{Z_1}\left[ f \right]\right] \\
&=\frac{1}{t}E_{Z}\left[Z^\T\nabla_{z}\left\{E_{Z_1}\left[f\right]\right\}^{\beta-1} 
\right] \\
&=\frac{1}{t}E_{Z}\left[\Delta_{z}\left\{E_{Z_1}\left[f\right]\right\}^{\beta-1} 
\right] \\
&=tE_{Z}\left[\Delta_{w}\left\{E_{Z_1}\left[f\right]\right\}^{\beta-1} 
\right],
  \end{split}
 \end{equation}
 where the fourth equality follows from the Gauss divergence theorem. 
By \eqref{eq:ggg.1}, \eqref{eq:ggg.2} and \eqref{eq:ggg.3}, we have
\begin{equation}\label{eq:ggg.4}
 \frac{\partial}{\partial t}g(t;w) 
 =
 tE_{Z}\left[\Delta_{w}\left\{E_{Z_1}\left[f\right]\right\}^{\beta-1} +
 (\beta-1)\left\{E_{Z_1}\left[f \right]\right\}^{\beta-2} 
 \Delta_{w}E_{Z_1}\left[  f \right]\right] .
\end{equation} 
Recall the formula of Laplacian for a function $h(u)$, 
\begin{align}\label{eq:rule_1}
 \Delta_u h^a(u)
 =a h^a(u)\left\{ \frac{\Delta_u h(u)}{h(u)}+(a-1)\|\nabla_u \log h(u)\|^2 \right\}, 
\end{align}
for $a\neq 0$.
Then, in \eqref{eq:ggg.4}, we have
\begin{equation}\label{eq:ggg.5}
 \begin{split}
& \Delta_{w}\left\{E_{Z_1}\left[f\right]\right\}^{\beta-1} +
 (\beta-1)\left\{E_{Z_1}\left[f \right]\right\}^{\beta-2} 
 \Delta_{w}E_{Z_1}\left[  f \right] \\
&=\frac{(\beta-1)}{\left\{E_{Z_1}\left[f\right]\right\}^{1-\beta}}
  \left(2\frac{\Delta_{w}E_{Z_1}\left[  f \right]}{E_{Z_1}\left[f\right]}
  +(\beta-2)\| \nabla_w \log E_{Z_1}\left[  f \right]\|^2\right) \\
&=\frac{2(\beta-1)}{\left\{E_{Z_1}\left[f\right]\right\}^{1-\beta}}
  \left(\frac{\Delta_{w}E_{Z_1}\left[  f \right]}{E_{Z_1}\left[f\right]}
  +\left(\beta/2-1\right)\| \nabla_w \log E_{Z_1}\left[  f \right]\|^2\right) \\
&=\frac{2(\beta-1)}{\left\{E_{Z_1}\left[f\right]\right\}^{1-\beta}}
\frac{\Delta_{w}\left\{E_{Z_1}\left[f\right]\right\}^{\beta/2}}{(\beta/2)\left\{E_{Z_1}\left[f\right]\right\}^{\beta/2}} \\
&=\frac{4(\beta-1)}{\beta}
\frac{\Delta_{w}\left\{E_{Z_1}\left[f\right]\right\}^{\beta/2}}{\left\{E_{Z_1}\left[f\right]\right\}^{1-\beta/2}}.
 \end{split}
\end{equation} 
By \eqref{IERD}, \eqref{eq:ggg.4} and \eqref{eq:ggg.5}, we completes the proof.
\end{proof}

\begin{remark}
 In the previous version of this article as well as \cite{George-etal-2006},
not only the Stein identity but also the heat equation
\begin{align*}
 \frac{\partial}{\partial v}\phi(u,v)
 =\frac{1}{2}\Delta_{u}\phi(u,v),
\end{align*}
 was efficiently applicable for deriving a nice expression of the risk difference,
 like Part \ref{1.thm:sufficient.1} of Theorem \ref{thm:sufficient.1}.
 It seemed to us that the heat equation was an additional necessary tool
 for investigating the Stein phenomenon of predictive density.
 But it is not true, the heat equation is no longer necessary.
 As seen in the proof of Theorem \ref{thm:sufficient.1},
 only the Stein identity or the Gauss divergence theorem
 is the key, as in Stein ``estimation'' problem.
\end{remark}

The superharmonicity of $f$ implies the superharmonicity of $E_{Z_1} \left[f(tZ_1+u)\right]$.
Furthermore, using the relationship \eqref{eq:rule_1}, we see that
the superharmonicity of $E_{Z_1} \left[f(tZ_1+u)\right]$ implies
the superharmonicity of
\begin{align*}
 \varrho(u;t;f)=\{E_{Z_1} \left[f(tZ_1+u)\right]\}^{\beta/2}
\end{align*}
for $\beta\in(0,1)$.
Hence, for Part \ref{2.thm:sufficient.1} of Theorem \ref{thm:sufficient.1},
we have a following corollary.
\begin{corollary}\label{cor:f}
 Suppose $f:\mathbb{R}^d\to\mathbb{R}_+$ is superharmonic.
 Then the predictive density $\hat{p}_f(y\mymid x; \alpha)$ given by \eqref{form:Bayes.predictive.density.f}
 as
 \begin{equation*}
  \hat{p}_f(y\mymid x; \alpha)
   =\frac{f(\gamma x+(1-\gamma)y)}
   {E_{Z_1}\left[f(x+\xi Z_1)\right]}\hat{p}_{\U}(y\mymid x;\alpha),
 \end{equation*}
 dominates $\hat{p}_{\U}(y\mymid x; \alpha)$.
\end{corollary}
In Section \ref{sec:harmonic}, we will investigate the properties of the Bayesian predictive density
$\hat{p}_\pi(y\mymid x; \alpha)$ where
\begin{align*}
 f(u)=\{m_\pi(u,\vxx\gamma)\}^{1/\beta}
\end{align*}
is assumed in Theorem \ref{thm:sufficient.1} and Corollary \ref{cor:f}.
Actually in this case, Corollary \ref{cor:f} is not useful
since the superharmonicity of $\{m_\pi(u,\vxx\gamma)\}^{1/\beta}$ for $\beta\in(0,1)$
is very restrictive.
Recall the relationship given by \eqref{eq:rule_1}. 
For example, the superharmonicity of $m_\pi(u,\vxx\gamma)$ does not imply
the superharmonicity of $\{m_\pi(u,\vxx\gamma)\}^{1/\beta}$.
Hence, in Section \ref{sec:harmonic}, we will seriously consider the superharmonicity
of
\begin{align*}
 \varrho(u;t;m_\pi^{1/\beta})=\left\{
 E_{Z_1} \left[\{m_\pi(t Z_1+u,\vxx\gamma)\}^{1/\beta}\right]\right\}^{\beta/2}.
\end{align*}

Further,
When $ 1/\beta=2/(1-\alpha)$ is not an integer, 
$ E_{Z_1} \left[\{m_\pi(t Z_1+u,\vxx\gamma)\}^{1/\beta}\right]$ in
Part \ref{2.thm:sufficient.1} of Theorem \ref{thm:sufficient.1}
is not tractable for our current methodology in Section \ref{sec:harmonic}.
Thus we propose a variant of 
Theorem \ref{thm:sufficient.1} with $  f(u)=\{m_\pi(u,\vxx\gamma)\}^{1/\beta}$,
 for a non-integer $ 1/\beta$ as follows.
Let $\kappa$ be the smallest integer among integers which is strictly greater than $ 1/\beta$, 
\begin{equation}
 \kappa
  = \min\{n\in\mathbb{Z}\mid n> 1/\beta\}. 
\end{equation}
Then $ \kappa -1 < 1/\beta< \kappa$.
As in \eqref{IERD}, 
the risk difference is expressed as
\begin{align*}
 & R_\alpha\{\phi(y-\mu,\vyy )\bD \hat{p}_{\U}(y\mymid x; \alpha) \}-
 R_\alpha\{\phi(y-\mu,\vyy ) \bD \hat{p}_\pi(y\mymid x; \alpha)\} \\
&= \frac{\gamma^{(1-\beta)d/2}}{\beta(1-\beta)}
  E_{W,Z}\left[E_{Z_1}\left[\left\{\frac{m_\pi(W+\xi(Z_1+Z),\vxx\gamma)}{m_\pi(W,\vxx\gamma)}
 \right\}^{1/\beta}\right]^{\beta-1}-1\right]
\end{align*}
where $W\sim N_d(\mu,\vxx \gamma I)$, $Z_1\sim N_d(0,I)$ and $Z\sim N_d(0,I)$.
From Jensen's inequality, 
we have
\begin{equation}\label{eq:JENSEN}
 \begin{split}
  & E_{Z_1}\left[m^{1/\beta}_\pi(w+\xi(Z_1+Z),\vxx\gamma)\right] \\
  & =E_{Z_1}\left[\{m^\kappa_\pi(w+\xi(Z_1+Z),\vxx\gamma)\}^{1/(\beta\kappa)}\right] \\
  & \leq \left\{E_{Z_1}\left[m^{\kappa}_\pi(w+\xi(Z_1+Z),\vxx\gamma)\right]\right\}^{1/(\beta\kappa)},
 \end{split}
\end{equation}
since $ 0< 1/(\beta\kappa)< 1$ and hence
\begin{align*}
 & R_\alpha\{\phi(y-\mu,\vyy )\bD \hat{p}_{\U}(y\mymid x; \alpha) \}-
 R_\alpha\{\phi(y-\mu,\vyy ) \bD \hat{p}_\pi(y\mymid x; \alpha)\} \\
 &\geq \frac{\gamma^{(1-\beta)d/2}}{\beta(1-\beta)}E_{W,Z}\left[
 E_{Z_1}\left[\frac{m^\kappa_\pi(W+\xi(Z_1+Z),\vxx\gamma)}{m^\kappa_\pi(W,\vxx\gamma)}
 \right]^{(\beta-1)/(\beta\kappa)}-1\right]. 
\end{align*}
Applying the same technique starting \eqref{IERD} through \eqref{eq:ggg.5}
to the lower bound above,
 we have a variant of Part \ref{2.thm:sufficient.1} of Theorem \ref{thm:sufficient.1}.
\begin{thm}\label{thm:sufficient.1.not}
 Assume $ 1/\beta$ is not a positive integer. 
 Let $\kappa$ be the smallest integer greater than $ 1/\beta$.
 A sufficient condition for $\mathrm{diff}R_{\alpha,\U,\pi}\geq 0$ is
 \begin{align}\label{eq:suffsuff.not}
  \Delta_u\left\{E_{Z_1} \left[m^\kappa_\pi(t Z_1+u,\vxx\gamma)\right]\right\}^{c(\beta)/\kappa}
  \leq 0,\quad\forall u\in\mathbb{R}^d, \ 0\leq \forall  t \leq \xi
 \end{align}
 where $Z_1\sim N_d(0,I)$  and
 \begin{equation}\label{cofbeta}
  c(\beta)
   =\frac{\kappa-1/\beta+1}{2}\in(1/2,1).
 \end{equation}
\end{thm}

\subsection{Asymptotics}
In this subsection, using Theorem \ref{thm:sufficient.1} with $f=m^{1/\beta}_\pi$,
we investigate asymptotics of the risk difference
\begin{equation*}
 \mathrm{diff}R_{\alpha,\U,\pi}
  = R_\alpha\left\{\phi(y-\mu,\vyy )\bD \hat{p}_{\U}(y\mymid x ; \alpha) \right\}-
  R_\alpha\left\{\phi(y-\mu,\vyy )\bD \hat{p}_{\pi}(y\mymid x ; \alpha) \right\}
\end{equation*}
where $ \hat{p}_{\U}(y\mymid x ; \alpha)$ and $\hat{p}_{\pi}(y\mymid x ; \alpha)$
are given by \eqref{bestequiv} and \eqref{form:Bayes.predictive.density}, respectively.

   \subsubsection{$\alpha\to -1$}
Let $v_*=\vxx\vyy/(\vxx+\vyy)$.
When $\alpha\to -1$ or equivalently $\beta\to 1$, we have
\begin{equation*}
 \gamma\to \frac{1}{1+\vxx/\vyy}=
\frac{v_*}{\vxx}
\text{ and }
  \xi^2 \to 
  \frac{\vxx^2}{\vxx+\vyy}=\vxx - v_*
\end{equation*}
and hence
\begin{equation}\label{asym.alpha.-1.1}
 \frac{2\gamma^{(1-\beta)d/2}}{\beta^2}\{m_\pi(w,\vxx\gamma)\}^{1/\beta-1}\to 2,  
\end{equation}
which are parts of $\rho(w,z)$ given by \eqref{eq:rho}.
Further, in $\varrho(t;u)$ given by \eqref{eq:varrho}, we have
\begin{equation}\label{eq:kitaichi.0}
 E_{Z_1} \left[m_\pi(t Z_1+u,\vxx \gamma)\right]
  =m_\pi(u,\vxx \gamma+t^2)\to m_\pi(u,v_*+t^2).
\end{equation}
By \eqref{asym.alpha.-1.1} and \eqref{eq:kitaichi.0},
we have
\begin{equation}\label{asym.alpha.-1.2}
\begin{split}
 \varrho(t;u)&\to m^{1/2}_\pi(u,v_*+t^2), \\
E_Z[\rho(w,Z)]&\to 4\int_0^{\sqrt{\vxx-v_*}}\int\mbd
 t\frac{-\Delta_u m^{1/2}_\pi(u,v_*+t^2)}{m^{1/2}_\pi(u,v_*+t^2)}\phi(u-w,t^2)\rd u  \rd t \\
 &=2\int_0^{\vxx-v_*}\int\mbd
 \frac{-\Delta_u m^{1/2}_\pi(u,v_*+t)}{m^{1/2}_\pi(u,v_*+t)}\phi(u-w,t)\rd u  \rd t .
\end{split}
\end{equation}
By \eqref{asym.alpha.-1.2}, we have
\begin{align*}
 E_{W,Z}[\rho(W,Z)]
 &\to 2\int_{\mathbb{R}^d}\left(\int_0^{\vxx-v_*}\int\mbd
 \frac{-\Delta_u m^{1/2}_\pi(u,v_*+t)}{m^{1/2}_\pi(u,v_*+t)}\phi(u-w,t)\rd u \rd t\right) \\
 &\qquad \qquad \times \phi(w-\mu,v_*)\rd w  \\
 &=
 2\int_0^{\vxx-v_*}\left(\int_{\mathbb{R}^d}
 \frac{-\Delta_u m^{1/2}_\pi(u,v_*+t)}{m^{1/2}_\pi(u,v_*+t)} 
 \phi(u-\mu,v_*+t) \rd u  \right)\rd t \\
 &=2\int_{v_*}^{\vxx}E_Z\left[-\frac{\Delta_z m^{1/2}_\pi(Z,v)}{m^{1/2}_\pi(Z,v)}\right]\rd v \\
 &=R_{-1}\{\phi(y-\mu,\vyy )\bD \hat{p}_{\U}(y\mymid x;-1) \}
 -R_{-1}\{\phi(y-\mu,\vyy )\bD \hat{p}_\pi(y\mymid x;-1) \},
\end{align*}
where $Z\sim N_d(\mu,vI)$ and $v_*=\vxx\vyy/(\vxx+\vyy)$. The last equality follows from 
\citeapos{George-etal-2006} result which was already explained in
\eqref{eq:beautiful} of Section \ref{sec:intro}. Hence we have
\begin{align*}
   \lim_{\alpha \to -1+0}\mathrm{diff}R_{\alpha,\U,\pi}
  =\mathrm{diff}R_{-1,\U,\pi}.
\end{align*}

\subsubsection{$(1-\alpha)\vxx/\vyy \to 0$}
Consider the asymptotic situation where
\begin{align}\label{eq:limit}
 (1-\alpha)\vxx/\vyy
 \to 0 \Leftrightarrow \beta(\vxx/\vyy)\to 0 \Leftrightarrow  \gamma\to 1.
\end{align}
Note that $ E_Z[\rho(w,Z)]$ is rewritten as the product $\rho_1(w)\rho_2(w)$
where
\begin{align*}
 \rho_1(w)
 &=\frac{2\gamma^{(1-\beta)d/2}}{\beta^2}\{m_\pi(w,\vxx\gamma)\}^{1/\beta-1}\xi^2, \\
 \rho_2(w)
 &= \frac{2}{\xi^2}\int_0^{\xi} t
 \left\{\int\mbd\frac{-\Delta_u\varrho(t;u)}{\varrho^{2/\beta-1}(t;u)}\phi(u-w,t^2)\rd u  \right\}\rd t \\
 &= \frac{1}{\xi^2}\int_0^{\xi^2} 
 \left\{\int\mbd\frac{-\Delta_u\varrho(\sqrt{t};u)}{\varrho^{2/\beta-1}(\sqrt{t};u)}\phi(u-w,t)\rd u  \right\}\rd t. 
\end{align*}
Since $\xi^2$ is rewritten as
\begin{equation}\label{asym.alpha.1.1}
 \xi^2=\frac{(1-\gamma)^2\vyy}{\gamma}=\left(\frac{1-\gamma}{\gamma}\right)^2\vyy \gamma
 =\frac{\vxx^2}{\vyy}\beta^2\gamma,
\end{equation}
we have
\begin{align*}
 \rho_1(w)
 =2\frac{\vxx^2}{\vyy}\gamma^{(1-\beta)d/2+1}\{m_\pi(w,\vxx\gamma)\}^{1/\beta-1}
\end{align*}
and
\begin{equation}\label{asym.alpha.1.2}
 \lim_{\gamma\to 1}\rho_1(w)=2\frac{\vxx^2}{\vyy}\{m_\pi(w,\vxx)\}^{1/\beta-1}.
\end{equation}
When $\gamma\to 1$, we have $\xi^2\to 0$ by \eqref{asym.alpha.1.1} and hence
\begin{equation}\label{asym.alpha.1.3}
 \begin{split}
 \lim_{\gamma\to 1} \rho_2(w)
  &=\lim_{t\to 0}
 \left\{\int\mbd\frac{-\Delta_u\varrho(\sqrt{t};u)}{\varrho^{2/\beta-1}(\sqrt{t};u)}\phi(u-w,t)\rd u  \right\} \\
  &=\int\mbd
  \lim_{t\to 0}
  \left(\frac{-\Delta_u\varrho(\sqrt{t};u)}{\varrho^{2/\beta-1}(\sqrt{t};u)}\right)\delta(u-w)\rd u ,
\end{split}
\end{equation}
where $\delta(\cdot)$ is the Dirac delta function.
By \eqref{asym.alpha.1.3} and
\begin{align*}
 \lim_{\substack{t\to 0\\\gamma\to 1}}\varrho(\sqrt{t};u)
 =\left\{\int\mbd m^{1/\beta}_\pi(u_1+u,\vxx\gamma)\delta(u_1)\rd u_1\right\}^{\beta/2} 
=m_\pi^{1/2}(u,\vxx),
\end{align*}
we have
\begin{align}\label{asym.alpha.1.4}
 \lim_{\gamma\to 1} \rho_2(w)
 =\left(-\Delta_w m_\pi^{1/2}(w,\vxx)\right)m_\pi^{1/2-1/\beta}(w,\vxx).
\end{align}
By \eqref{asym.alpha.1.2} and \eqref{asym.alpha.1.4}, we have
\begin{align*}
 \lim_{\gamma\to 1} E_Z[\rho(w,Z)]
 =\lim_{\gamma\to 1} \rho_1(w)\rho_2(w)
 =2\frac{\vxx^2}{\vyy}\frac{-\Delta_w m_\pi^{1/2}(w,\vxx)}{m_\pi^{1/2}(w,\vxx)},
\end{align*}
which implies that
\begin{align*}
 \lim_{\alpha\to 1}\mathrm{diff}R_{\alpha,\U,\pi}
 &=\mathrm{diff}R_{1,\U,\pi}
 =2\frac{\vxx^2}{\vyy}E\left[\frac{-\Delta_w m_\pi^{1/2}(W,\vxx)}{m_\pi^{1/2}(W,\vxx)}\right], \\
 \lim_{\vxx/\vyy\to 0}\frac{\vyy}{\vxx}\mathrm{diff}R_{\alpha,\U,\pi}
 &=\frac{\vyy}{\vxx}\mathrm{diff}R_{1,\U,\pi}
 =2\vxx E\left[\frac{-\Delta_w m_\pi^{1/2}(W,\vxx)}{m_\pi^{1/2}(W,\vxx)}\right].
\end{align*}
Therefore the asymptotic situation $ \vxx/\vyy \to 0$ corresponds to the case $\alpha\to 1$ and 
$ \Delta_z\{m_\pi(z,v)\}^{1/2}$ plays an important role for general $\alpha\in(-1,1)$.

\section{Improvement under the harmonic prior}
\label{sec:harmonic}
Under the harmonic prior $\pi_\HH(\mu)=\|\mu\|^{-(d-2)}$, let
\begin{equation}
  m_\HH(w,v)=\int\mbd\phi(w-\mu,v)\pi_\HH(\mu)\rd\mu . 
\end{equation}
Let $ \nu$ be an integer larger than or equal to $2$.
The superharmonicity related to $E_{Z_1}\left[m^\nu_\HH(t Z_1+u,v)\right]$
with $Z_1\sim N_d(0,I)$ is as follows.
\begin{thm}\label{thm:sufficient.harmonic}
Let $c\in(0,1)$ and $Z_1\sim N_d(0,I)$. 
 Let $\nu$ be an integer larger than or equal to $2$.
 Then, we have  
\begin{equation*}
 \Delta_u\left\{E_{Z_1} \left[m^\nu_\HH( t Z_1+u,v)\right]\right\}^{c/\nu}\leq 0,
  \quad\forall u\in\mathbb{R}^d,
\end{equation*}
 when
\begin{equation}\label{thm:sufficient.harmonic.t}
0\leq  t\leq \left(\frac{(d+2)(1-c)v}{d\nu(\nu-1)}\right)^{1/2}.
\end{equation}
\end{thm}
\begin{proof}
 Section \ref{sec:proof.harmonic} of Appendix.
\end{proof}
When $ 1/\beta$ is an integer larger than or equal to $2$,
namely,
\begin{equation}
\begin{split}
 \alpha&=0,1/3, 1/2, 3/5, 2/3, \dots, \\
 \beta&=1/2,1/3, 1/4, 1/5, 1/6, \dots,
\end{split} 
\end{equation}
let $\nu=1/\beta$, $v=\vxx\gamma$ and $c=1/2$ in Theorem \ref{thm:sufficient.harmonic}
and compare \eqref{thm:sufficient.harmonic.t} in  Theorem \ref{thm:sufficient.harmonic}
with $0\leq t^2\leq \xi^2=\beta^2\vxx^2\gamma/\vyy$ in Theorem \ref{thm:sufficient.1}.
If 
\begin{equation*}
 \frac{\beta^2\vxx}{\vyy}\vxx\gamma \leq
  \frac{(d+2)(1-c)}{d\nu(\nu-1)}\vxx\gamma
\end{equation*}
or equivalently
\begin{equation*}
\frac{\vxx}{\vyy}\leq \frac{d+2}{d(1+\alpha)}=\frac{d+2}{2d(1-\beta)},
\end{equation*}
$m_\HH(w,\vxx\gamma)$ satisfies the sufficient condition of Theorem \ref{thm:sufficient.1}
and we have a following result of the Bayesian predictive density with respect to
Stein's harmonic prior $\pi_\HH(\mu)=\|\mu\|^{-(d-2)}$, which is given by
 \begin{equation}\label{form:Bayes.predictive.density.harmonic}
  \hat{p}_\HH(y\mymid x; \alpha)
  =\frac{m^{1/\beta}_\HH(\gamma x+(1-\gamma)y,\vxx\gamma)}
{E_{Z_1}\left[m^{1/\beta}_\HH(x+\xi Z_1,\vxx\gamma)\right]}\hat{p}_{\U}(y\mymid x;\alpha).
 \end{equation}

\begin{thm}\label{thm:mainthm}
 Suppose $2/(1-\alpha)$ is an positive integer for $\alpha\in(-1,1)$.
 Suppose
\begin{equation}\label{vy/vx}
 \frac{\vxx}{\vyy} \leq \frac{d+2}{d(1+\alpha)}.
\end{equation}
 Then, under $\alpha$-div loss,
 the Bayesian predictive density $\hat{p}_\HH(y\mymid x;\alpha)$
 with respect to the harmonic prior $\pi_\HH(\mu)=\|\mu\|^{-(d-2)}$
 dominates the best invariant Bayesian predictive density
 $\hat{p}_\U(y\mymid x;\alpha)=\phi(y-x,\vyy /\gamma)$.
\end{thm}
\begin{remark}
 For any $d\geq 3$ and $\alpha\in(-1,1)$, we have
\begin{equation*}
 \frac {d+2}{d(1+\alpha)}> \frac{1}{2}.
\end{equation*} 
 Note that, in most typical situations,
\begin{equation*}
 \frac{\vxx}{\vyy} \leq \frac{1}{2}, 
\end{equation*}
is easily assumed as follows.
Suppose that we have a set of observations $x_1,\dots,x_n$ from $N_d(\mu,\sigma^2I)$.
An unobserved set $x_{n+1},\dots,x_{n+m}$ from the same distribution
is predicted by using a predictive density as a function of $x_1,\dots,x_n$.
From sufficiency,
\begin{align*}
 x=n^{-1}\sum\nolimits_{i=1}^n x_i\sim N_d(\mu,\sigma^2I/n)
 \text{ and }
 y=m^{-1}\sum\nolimits_{i=1}^m x_{n+i}\sim N_d(\mu,\sigma^2I/m)
\end{align*}
and clearly $ \vxx/\vyy =m/n$ in this case. 
Since, $m$ is typically $1$ or $2$ whereas $n$ is relatively large, the condition \eqref{vy/vx}
is satisfied.
\end{remark}
\smallskip
When $\beta=2/(1-\alpha)$ is not an integer, Theorem \ref{thm:sufficient.1.not} can be applied.
Let $\kappa$ be the smallest integer greater than $ 1/\beta$.
Suppose 
\begin{equation}\label{inequ.non}
 \beta^2\frac{\vxx}{\vyy}\vxx\gamma \leq
  \frac{(d+2)\{1-c(\beta)\}\vxx\gamma}{d\kappa(\kappa-1)},
\end{equation}
where $c(\beta)$ is given by \eqref{cofbeta} as
$c(\beta)=c(\{1-\alpha\}/2)=\{\kappa-2/(1-\alpha)+1\}/2$,
the left-hand side is the upper bound of $t$ of Theorem \ref{thm:sufficient.1.not} and
the right-hand side is the upper bound of $t$ of Theorem \ref{thm:sufficient.harmonic}.
When 
\begin{equation*}
 \frac{\vxx}{\vyy} \leq 
  \left(\frac{2}{1-\alpha}\right)^2\frac{d+2}{d}
  \frac{1-\left\{\kappa-2/(1-\alpha)\right\}}{2\kappa(\kappa-1)},
\end{equation*}
which is equivalent to \eqref{inequ.non},
$m_\HH(w,\vxx\gamma)$ satisfies the sufficient condition of Theorem \ref{thm:sufficient.1.not}
and we have a following result.
\begin{thm}\label{thm:mainthm.not}
 Suppose $2/(1-\alpha)$ is not an positive integer for $\alpha\in(-1,1)$.
 Let $\kappa$ be the smallest integer greater than $ 2/(1-\alpha)$.
 Suppose
 \begin{equation}\label{vy/vx.not}
  \frac{\vxx}{\vyy} \leq
   \left(\frac{2}{1-\alpha}\right)^2
   \frac{d+2}{d}
   \frac{1-\left\{\kappa-2/(1-\alpha)\right\}}{2\kappa(\kappa-1)}.
 \end{equation}
 Then the Bayesian predictive density $\hat{p}_\HH(y\mymid x;\alpha)$
 with respect to the harmonic prior $\pi_\HH(\mu)=\|\mu\|^{-(d-2)}$
 dominates the best invariant Bayesian predictive density
 $\hat{p}_\U(y\mymid x;\alpha)=\phi(y-x,\vyy /\gamma)$.
\end{thm}
By the definition of $\kappa$,  
\begin{equation*}
 \kappa-1<\frac{2}{1-\alpha}<\kappa.
\end{equation*}
As $ 2/(1-\alpha)\uparrow \kappa$, the upper bound given by \eqref{vy/vx.not} approaches
$ (d+2)/\{d(1+\alpha)\}$ which is exactly the upper bound  given by \eqref{vy/vx} of Theorem \ref{thm:mainthm}.
On the other hand, as $ 2/(1-\alpha)\downarrow \kappa-1$, the upper bound given by \eqref{vy/vx.not}
approaches $0$. Figure \ref{fig:1} gives a graph of behavior of the upper bound of
$\vxx/\vyy$ for improvement in Theorems \ref{thm:mainthm} and \ref{thm:mainthm.not}.
This undesirable discontinuity with respect to the upper bound of Theorem \ref{thm:mainthm.not}
is due to Jensen's inequality \eqref{eq:JENSEN} which was not used in the proof of
Theorem \ref{thm:sufficient.1}.
However, we would like to emphasize that, for any $\alpha\in(-1,1)$, there exists a positive
upper bound of of $\vxx/\vyy$ for improvement.
We can naturally make a conjecture that the lower bound of $\vyy/\vxx$ for improvement,
$ d(1+\alpha)/(d+2)$, of Theorem \ref{thm:mainthm} is still valid even if 
$2/(1-\alpha)$ is not an integer.
For that purpose, the methodology for appropriately treating
$E_{Z_1} \left[\{m_\HH( t Z_1+u,\vxx\gamma)\}^{2/(1-\alpha)}\right]$
or more generally $E_{Z_1} \left[\{m_\pi(t Z_1+u,\vxx\gamma)\}^{2/(1-\alpha)}\right]$
for non-integer $2/(1-\alpha)$ is needed and it remains an open problem.

\begin{figure}
\centering
 \scalebox{0.42}{\input{dnorm.tex}}
\caption{The upper bound of $\vxx/\vyy$ in Theorems \ref{thm:mainthm} and \ref{thm:mainthm.not}}
\label{fig:1}
\end{figure}

\appendix
\section{Minimaxity of $\hat{p}_{\U}(y\mymid x ; \alpha)$}
\label{sec:minimaxity}
In this section, we show that
\begin{equation}
 \hat{p}_{\U}(y\mymid x ; \alpha)
  =\phi(y-x,\vyy /\gamma)=\phi(y-x,\vyy+\beta\vxx)
\end{equation}
is minimax, by following Sections II and III of \cite{Liang-Barron-2004}. 
We start with the definition of invariance under location shift.
\begin{definition}
A predictive density $\hat{p}(y\mymid x)$ is invariant under location shift,
if for all $a\in\mathbb{R}^d$ and all $x$, $y$,
$\hat{p}(y+a \mymid x+a)=\hat{p}(y\mymid x)$.
\end{definition}
Hence any invariant predictive density should be of the form
\begin{align*}
 \hat{p}(y\mymid x)=q(y-x)
\end{align*}
which satisfies
\begin{align*}
 \int_{\mathbb{R}^d}q(y)\rd y =1.
\end{align*}
Clearly $ \hat{p}_{\U}(y\mymid x ; \alpha)$ is invariant under location shift.
Note that invariant procedures have constant risk since
the risk of the invariant predictive density $q(y-x)$ is
\begin{equation}\label{eq:risk.q}
 \begin{split}
& R_\alpha\{\phi(y-\mu,\vyy )\bD q(y-x) \} \\
& =\int_{\mathbb{R}^d}\left(\int_{\mathbb{R}^d} f_\alpha\left(\frac{q(y-x)}{\phi(y-\mu,\vyy)}\right)\phi(y-\mu,\vyy)\rd y \right)
 \phi(x-\mu,\vxx)\rd x  \\
& = \int_{\mathbb{R}^d}\left(\int_{\mathbb{R}^d}
 f_\alpha\left(\frac{q(z_y-z_x)}{\phi(z_y,\vyy)}\right)\phi(z_y,\vyy)\rd z_y\right)
 \phi(z_x,\vxx)\rd z_x
\end{split}
\end{equation}
where $z_x=x-\mu$  and $z_y=y-\mu$, which does not depend on $\mu$.
More specifically,
the risk of the invariant predictive density $q(y-x)$ is as follows.
\begin{lemma}\label{lem:risk_invariance}
The risk of an invariant predictive density $q(y-x)$ is
\begin{equation}\label{eq:lem:risk_invariance}
 \begin{split}
& R_\alpha\{\phi(y-\mu,\vyy )\bD q(y-x) \} \\
&= \frac{1-\gamma^{(1-\beta)d/2}}{\beta(1-\beta)}+\gamma^{(1-\beta)d/2}
D_{\alpha}\left\{\phi(z,\vyy/\gamma )\bD q(z)\right\}.
\end{split}
\end{equation}
\end{lemma}
\begin{proof}
By \eqref{eq:risk.q} and the definition of $\alpha$-div loss,
\begin{align*}
& R_\alpha\{\phi(y-\mu,\vyy )\bD q(y-x) \} \\
 &= \frac{1}{\beta(1-\beta)}
 \left\{
 1- \int_{\mathbb{R}^d}\int_{\mathbb{R}^d}q^{1-\beta}(y-x)\phi^\beta(y,\vyy)\phi(x,\vxx)\rd x \rd y 
 \right\} .
\end{align*}
By the identity \eqref{eq:hirei} with $\mu=0$, we have
\begin{equation*}
  \phi\left(x,\vxx \right)\phi^\beta( y,\vyy ) 
  =\gamma^{(1-\beta)d/2}\phi(\gamma x+(1-\gamma)y,\vxx \gamma)\phi^\beta(y-x,\vyy /\gamma),
\end{equation*}
 and hence
 \begin{align*}
 R_\alpha\{\phi(y-\mu,\vyy )\bD q(y-x) \} 
 &= \frac{1}{\beta(1-\beta)}
 \left\{
 1- \gamma^{(1-\beta)d/2}\int_{\mathbb{R}^d}\int_{\mathbb{R}^d}
 q^{1-\beta}(y-x)\right. \\
 &\quad \times \left. \phi^\beta(y-x,\vyy/\gamma)\phi(\gamma x+(1-\gamma)y,\vxx\gamma)\rd x \rd y 
 \right\}.
 \end{align*}
By the change of variables,
\begin{equation}
 \begin{pmatrix}
  w \\ z 
 \end{pmatrix}
=\begin{pmatrix}
  \gamma I_d & (1-\gamma)I_d \\ - I_d & I_d
 \end{pmatrix}
 \begin{pmatrix}
  x \\ y
 \end{pmatrix}
\end{equation}
where Jacobian of the matrix is $1$, we have
\begin{align*}
& R_\alpha\{\phi(y-\mu,\vyy )\bD q(y-x) \} \\
 &= \frac{1}{\beta(1-\beta)}
 \left\{
 1- \gamma^{(1-\beta)d/2}\int_{\mathbb{R}^d}\int_{\mathbb{R}^d}
 q^{1-\beta}(z)\phi^\beta(z,\vyy/\gamma)\phi(w,\vxx\gamma)\rd z\rd w 
 \right\} \\
 &= \frac{1}{\beta(1-\beta)}
 \left\{
 1- \gamma^{(1-\beta)d/2}\int_{\mathbb{R}^d}
 q^{1-\beta}(z)\phi^\beta(z,\vyy/\gamma)\rd z
 \right\} \\
 &= \frac{1-\gamma^{(1-\beta)d/2}}{\beta(1-\beta)}+\gamma^{(1-\beta)d/2}
D_{\alpha}\left\{\phi(z,\vyy/\gamma )\bD q(z)\right\}.
\end{align*}
\end{proof}
In \eqref{eq:lem:risk_invariance} of Lemma \ref{lem:risk_invariance}, 
$D_{\alpha}\left\{\phi(z,\vyy/\gamma )\bD q(z)\right\}$ is non-negative and
takes zero if and only if $q(z)=\phi(z,\vyy/\gamma )$.
Hence the best invariant procedure is $ \hat{p}_{\U}(y\mymid x ; \alpha)=\phi(y-x,\vyy/\gamma )$,
where the constant risk is
\begin{align*}
 \frac{1-\gamma^{(1-\beta)d/2}}{\beta(1-\beta)}.
\end{align*}
Since the risk is constant for invariant predictive density,
the best invariant $ \hat{p}_{\U}(y\mymid x ; \alpha)$
is the minimax procedure among all invariant procedures.
If a constant risk procedure is shown to have an extended Bayes property defined in the below,
then it is, in fact, minimax over all procedures. See Theorem 5.18 of \cite{Berger-1985}
and Theorem 5.1.12 of \cite{Lehmann-Casella-1998} for the detail.
\begin{definition}
 A predictive procedure $\hat{p}_*(y\mymid x)$ is called extended Bayes,
 if there exists a sequence of Bayes
 procedures $\hat{p}_{\pi_c}(y\mymid x; \alpha)$ with proper prior densities $\pi_c(\mu)$ for $c=1,\dots,$
 such chat their Bayes risk differences go to zero,
 that is,
 \begin{align*}
  &\lim_{c\to\infty}\left(
  \int_{\mathbb{R}^d}R_\alpha\{\phi(y-\mu,\vyy )\bD  \hat{p}_*(y\mymid x)\}\pi_c(\mu)\rd\mu  \right.\\
  & \quad\qquad \left.
-\int_{\mathbb{R}^d}R_\alpha\{\phi(y-\mu,\vyy )\bD  \hat{p}_{\pi_c}(y\mymid x; \alpha)\}\pi_c(\mu)\rd\mu \right)=0.
 \end{align*}
\end{definition}
%
Recall that
\begin{equation} \label{gene-B-density-app}
 \hat{p}_{\pi}(y\mymid x; \alpha) \propto
 \left\{\int\mbd \phi^\beta(y-\mu,\vyy )\phi(x-\mu,\vxx )\pi(\mu)\rd\mu  \right\}^{1/\beta}
\end{equation}
for $\beta=(1-\alpha)/2$ and $\alpha\in(-1,1)$.
Under the prior $\mu\sim N_d(0,\{c\vxx\gamma\}I)$ with the density 
$\pi_{c}(\mu)=\phi(\mu,c\vxx\gamma)$, the Bayesian solution is
\begin{align*}
 \hat{p}_{\pi_c}(y\mymid x;\alpha)=\phi\left(y-\frac{c\gamma}{1+c\gamma}x,\vyy\frac{1+c}{1+c\gamma}\right)
\end{align*}
by 
the identity
\begin{equation}\label{eq:hirei.1}
 \begin{split}
  &\phi^\beta( y -\mu,\vyy )\phi\left(x-\mu,\vxx \right)\phi(\mu,c\vxx\gamma) \\      
  &=
\left(\frac{1+c\gamma}{1+c}\right)^{d(1-\beta)/2}
  \phi\left(\mu-c\frac{\gamma x+(1-\gamma)y}{1+c},\frac{c\vxx \gamma}{1+c}\right)\\
&\qquad\times\phi^\beta\left(y-\frac{c\gamma x}{1+c\gamma},\vyy\frac{1+c}{1+c\gamma}\right)
  \phi\left(x,\vxx(1+c\gamma)\right). 
 \end{split}
\end{equation}
Furthermore, by the identity \eqref{eq:hirei.1}, the Bayes risk of
$ \hat{p}_{\pi_c}(y\mymid x;\alpha)$
is given by
\begin{align*}
& \frac{1}{\beta(1-\beta)}
 \biggl(1-\int_{\mathbb{R}^d}\int_{\mathbb{R}^d}\int_{\mathbb{R}^d}
 \left\{\frac{\hat{p}_{\pi_c}(y\mymid x;\alpha)}{\phi(y-\mu,\vyy)}\right\}^{1-\beta}
  \\
&\qquad \qquad \qquad 
\times\phi\left(x-\mu,\vxx \right)\phi( y -\mu,\vyy )\phi(\mu,c\vxx\gamma)\rd x \rd y \rd\mu \biggr) \\
& =\frac{1}{\beta(1-\beta)}
 \left\{1-\left(\frac{1+c\gamma}{1+c}\right)^{d(1-\beta)/2}
 \int_{\mathbb{R}^d}\int_{\mathbb{R}^d}\int_{\mathbb{R}^d}
  \phi\left(\mu-c\frac{\gamma x+(1-\gamma)y}{1+c},\frac{c\vxx \gamma}{1+c}\right) \right. \\
&\qquad\times\left. \phi\left(y-\frac{c\gamma x}{1+c\gamma},\vyy\frac{1+c}{1+c\gamma}\right)
  \phi\left(x,\vxx(1+c\gamma)\right) \rd\mu  \rd y  \rd x  \right\}\\ 
 &=\frac{1}{\beta(1-\beta)}\left\{1-\left(\frac{1+c\gamma}{1+c}\right)^{d(1-\beta)/2}\right\},
\end{align*}
which approaches $ (1-\gamma^{(1-\beta)d/2})/\{\beta(1-\beta)\}$ as $c$ goes to infinity,
the constant risk of $\hat{p}_{\U}(y\mymid x ; \alpha)$.
Hence $\hat{p}_{\U}(y\mymid x ; \alpha)$ is extended Bayes and hence minimax.

%

 \section{Proof of Theorem \ref{thm:sufficient.harmonic}}
\label{sec:proof.harmonic}
Recall  the identity 
\begin{equation}\label{beki}
 \|\mu\|^{-(d-2)}
  =b\int_0^\infty g^{d/2-2}\exp\left(-g\frac{\|\mu\|^2}{2v}\right)\rd g
\end{equation}
for any $v>0$, where $b=1/\{\Gamma(d/2-1)2^{d/2-1}v^{d/2-1}\}$.
Then  we have
      \begin{align*}
       m_\HH(w,v)
       &=\int\mbd\phi(w-\mu,v)\|\mu\|^{-(d-2)}\rd\mu  \\
       &=b\int_0^\infty g^{d/2-2} \rd g \int\mbd\frac{1}{(2\pi)^{d/2}v^{d/2}}
       \exp\left(-\frac{\|w-\mu\|^2}{2v}-g\frac{\|\mu\|^2}{2v}\right) \rd\mu  \\
       &=b\int_0^\infty\frac{g^{d/2-2}}{(1+g)^{d/2}}\exp\left(-\frac{g\|w\|^2}{2(g+1)v}\right)\rd g\\
       &=b\int_0^1\lambda^{d/2-2}\exp\left(-\frac{\lambda\|w\|^2}{2v}\right)\rd\lambda ,
      \end{align*}
where the third equality is from the relation of completing squares with respect to $\mu$
\begin{equation*}
 \|w-\mu\|^2+g\|\mu\|^2=(g+1)\|\mu-w/(g+1)\|^2+\{g/(g+1)\}\|w\|^2
\end{equation*}
and the fourth equality is from the transformation $\lambda=g/(g+1)$.

Note that $m^\nu_\HH(w,v)$ for a positive integer $\nu$ is expressed as
\begin{align*}
 m^\nu_\HH(w,v)
 = b^\nu \int_{\mathcal{D}_{\nu}}\prod_{i=1}^\nu \lambda_i^{d/2-2}
      \exp\left(-\frac{\sum_{i=1}^\nu \lambda_i \|w\|^2}{2v}\right)\prod \rd\lambda_i,
\end{align*}
where $\mathcal{D}_\nu$ is $\nu$-dimensional unit hyper-cube.
In the following, $ \rd\lambda $ denotes $\prod_{i=1}^\nu \rd\lambda_i$ for notational simplicity.
Furthermore the subscript and superscript of $\prod$ and $\sum$ is omitted for simplicity
if they are $i=1$ and $i=\nu$ respectively. Hence $m^\nu_\HH(w,v)$ in the above is written as
\begin{align*}
 m^\nu_\HH(w,v)
 = b^\nu \int_{\mathcal{D}_{\nu}}\prod \lambda_i^{d/2-2}
   \exp\left(-\frac{\sum \lambda_i \|w\|^2}{2v}\right)\rd\lambda .
\end{align*}
 
For the calculation of
\begin{equation}
 E_{Z_1} \left[m^\nu_\HH( t Z_1+u,v)\right]
  =\int\mbd m^\nu_\HH(x+u,v)\phi(x,t^2)\rd x 
\end{equation}
under $Z_1\sim N_d(0,I)$,
note the relation of completing squares with respect to $x$,
\begin{equation}\label{completing_squares_tau}
 \begin{split}
 & \frac{\left(\sum\lambda_i\right)\|x+u\|^2}{v}+\frac{\|x\|^2}{t^2} 
  =\frac{1}{v}\left\{\sum\lambda_i\|x+u\|^2+s\|x\|^2\right\} \\
 & =\frac{1}{v}\left\{\left(\sum\lambda_i+s\right)\left\|
 x+\frac{\sum\lambda_i}{\sum\lambda_i+s}u\right\|^2
 +\frac{s\sum\lambda_i}{\sum\lambda_i+s}\|u\|^2\right\},
\end{split}
\end{equation}
where $s=v/t^2$.
Then, by \eqref{completing_squares_tau}, we have
\begin{align*}
 E_{Z_1}[m^\nu_\HH(t Z_1+u,v)]
 =\frac{b^\nu v^{d/2}}{t^d}
 \int_{\mathcal{D}_{\nu}}\frac{\prod \lambda_i^{d/2-2}}{(\sum\lambda_i+s)^{d/2}} 
\exp\left(-\frac{s\sum\lambda_i}{v(\sum\lambda_i+s)}\frac{\|u\|^2}{2}\right)\rd\lambda .
\end{align*}
Re-define $u:=\{s/v\}^{1/2}u$ and let
\begin{equation}\label{psipsipsi}
 \psi(u;\nu,s)=\int_{\mathcal{D}_{\nu}}
\frac{\prod \lambda_i^{d/2-2}}{ (\sum\lambda_i+s)^{d/2}}
 \exp\left(-\frac{\sum\lambda_i}{\sum\lambda_i+s}\frac{\|u\|^2}{2}\right)\rd\lambda .
\end{equation}
By \eqref{eq:rule_1},
the super-harmonicity of $\left\{E_{Z_1} \left[m^\nu_\HH(t Z_1+u,v)\right]\right\}^{c/\nu}$
with respect to $u\in\mathbb{R}^d$ is equivalent to 
\begin{equation}\label{eq:suffsuffsuff}
 \left(\frac{c}{\nu}-1\right)\|\nabla_u \psi\|^2 + \psi\Delta_u\psi \leq 0,
 \quad \forall u\in\mathbb{R}^d.
\end{equation}
The integrand of $\psi$ given by \eqref{psipsipsi} is denoted by
\begin{equation*}
 \zeta(\lambda)
  = \zeta(\lambda_1,\dots,\lambda_\nu)
  =\frac{\prod \lambda_i^{d/2-2}}{ (\sum\lambda_i+s)^{d/2}}
 \exp\left(-\frac{\sum\lambda_i}{\sum\lambda_i+s}z\right)
\end{equation*}
where $z=\|u\|^2/2$. Then we have
\begin{equation*}
 \frac{\partial}{\partial u_j}\psi
  =-u_j\int \zeta(\lambda)\frac{\sum\lambda_i}{\sum\lambda_i+s}\rd\lambda , 
\end{equation*}
for $j=1,\dots,d$ and
\begin{align*}
 \frac{\partial^2}{\partial u_j^2}\psi 
 =\int \zeta(\lambda)\left\{-\frac{\sum\lambda_i}{\sum\lambda_i+s}
 +u_j^2\left(\frac{\sum\lambda_i}{\sum\lambda_i+s}\right)^2\right\}\rd\lambda .
\end{align*}
Noting $z=\|u\|^2/2$, we have
\begin{equation}\label{nabla_psi}
 \|\nabla_u \psi\|^2
  =2z\left(\int \zeta(\lambda) \frac{\sum\lambda_i}{\sum\lambda_i+s}\rd\lambda \right)^2 
=2 \nu^2 z \left(\int \zeta(\lambda)\frac{\lambda_1}{\sum\lambda_i+s}\rd\lambda \right)^2
\end{equation}
and
\begin{equation}\label{lap_psi}
\begin{split}
 \Delta_u\psi&=
 -d\int \zeta(\lambda) \frac{\sum\lambda_i}{\sum\lambda_i+s}\rd\lambda 
  +2z\int \zeta(\lambda)\left(\frac{\sum\lambda_i}{\sum\lambda_i+s}\right)^2 \rd\lambda  \\
 & =-d\nu\int \zeta(\lambda)\frac{\lambda_1}{\sum\lambda_i+s}\rd\lambda 
  +2\nu z\int \zeta(\lambda)\frac{\lambda_1^2}{(\sum\lambda_i+s)^2}\rd\lambda  \\
 &\quad +2\nu(\nu-1)z\int \zeta(\lambda) \frac{\lambda_1\lambda_2}{(\sum\lambda_i+s)^2}\rd\lambda . 
\end{split} 
\end{equation}
In \eqref{nabla_psi} and \eqref{lap_psi},
the second equalities are from symmetry with respect to $\lambda_i$'s.

Let
\begin{align*}
 \rho(j_1,j_2,l)
 &= \int_{\mathcal{D}_\nu} \lambda_1^{j_1}\lambda_2^{j_2}(\sum\lambda_i+s)^{l}
 \zeta(\lambda) \rd\lambda , \\
 \eta(j_2,l)
 &= \int_{\mathcal{D}_{\nu-1}}  \lambda_2^{j_2}\left(1+\sum\nolimits_{i=2}\lambda_i+s\right)^{l}
 \zeta(1,\lambda_2,\dots,\lambda_\nu) \prod_{i=2} \rd\lambda_i,
\end{align*}
where $j_1$ and $j_2$ are nonnegative integers. Then $\|\nabla_u \psi\|^2$ and
$ \Delta_u \psi$ given by \eqref{nabla_psi} and \eqref{lap_psi} is rewritten as
\begin{equation}\label{psi2deltapsi}
\begin{split}
 \|\nabla_u \psi\|^2
 &=2\nu^2 z\rho(1,0,-1)^2, \\
 \Delta_u \psi
 &=-d\nu \rho(1,0,-1) +2\nu z \rho(2,0,-2)+2\nu(\nu-1)z \rho(1,1,-2).
\end{split}  
\end{equation}
Here are some useful relationships and inequalities.
\begin{lemma}\label{useful.lemma}
\begin{align}
 sz\rho(j_1,j_2,l)
 &=-\eta(j_2,l+2)+(j_1+d/2-2)\rho(j_1-1,j_2,l+2)  \notag \\
 &\quad +(l-d/2+2)\rho(j_1,j_2,l+1), \text{ for }j_1\geq 1, \label{rinsetsu.0} \\
 \rho(0,0,l)
 &=\nu\rho(1,0,l-1)+s\rho(0,0,l-1), \label{rinsetsu.1}\\
 \rho(1,0,l)
 &=\rho(2,0,l-1)+(\nu-1)\rho(1,1,l-1)+s\rho(1,0,l-1),\label{rinsetsu.2} \\
 \eta(0,1)
 &=\eta(0,0)+(\nu-1)\eta(1,0)+s\eta(0,0),\label{rinsetsu.3} \\
 \eta(0,1)\rho(0,0,-1)
 &\geq \eta(0,0)\rho(0,0,0),\label{FKG.inequality} \\
 \frac{\rho(1,0,-1)}{\rho(1,0,0)}
 &\geq \frac{1}{\nu d/(d+2) +s }.  \label{inequality.000}
\end{align}
\end{lemma}
\begin{proof}
 See Sub-section \ref{subsec:useful.lemma}.
\end{proof}
Applying the identity \eqref{rinsetsu.0} to  $\|\nabla_u \psi\|^2$ and $ \Delta_u \psi$
given in \eqref{psi2deltapsi}, we have
\begin{align*}
 s\|\nabla_u \psi\|^2
&=2\nu^2\{sz\rho(1,0,-1)\}\rho(1,0,-1) \\
 &=\nu^2\left\{-2\eta(0,1)+(d-2)\rho(0,0,1)-(d-2)\rho(1,0,0)\right\}\rho(1,0,-1), \\
 s\Delta_u \psi
 &=-d\nu s \rho(1,0,-1) +\nu\{-2\eta(0,0)+d\rho(1,0,0)-d\rho(2,0,-1)\} \\
& \qquad +\nu(\nu-1)\left\{-2\eta(1,0)+(d-2)\rho(1,0,0)-d\rho(1,1,-1)\right\} \\
&=-2\nu\{\eta(0,0)+(\nu-1)\eta(1,0)\}+\nu(\nu-1)(d-2)\rho(1,0,0),
\end{align*}
where the second equality of $s\Delta_u \psi$ follows from \eqref{rinsetsu.2}.
Then we have
\begin{align}
&\frac{s}{\nu}\left(\frac{c-\nu}{\nu}\|\nabla_u \psi\|^2+\psi \Delta_u \psi\right)  \label{eq:2skappa}\\
&=(\nu-c)\left[2\eta(0,1)-(d-2)\{\rho(0,0,1)-\rho(1,0,0)\}\right] \rho(1,0,-1) \notag \\
 &\quad -2\{\eta(0,0)+(\nu-1)\eta(1,0)\}\rho(0,0,0) +(\nu-1)(d-2)\rho(1,0,0)\rho(0,0,0).\notag
\end{align}
By applying \eqref{rinsetsu.1}, \eqref{rinsetsu.3} and \eqref{FKG.inequality},
the terms of \eqref{eq:2skappa} including $\eta(\cdot,\cdot)$, divided by $2$, is
\begin{equation}\label{terms.including.tau}
 \begin{split}
 &(\nu-c)\eta(0,1)\rho(1,0,-1)-\{\eta(0,0)+(\nu-1)\eta(1,0)\}\rho(0,0,0)\\
 &=(\nu-c)\eta(0,1)\rho(1,0,-1)-\{\eta(0,1)-s\eta(0,0)\}\rho(0,0,0) \\
 &=(\nu-c)\eta(0,1)\rho(1,0,-1)-\eta(0,1)\left\{\nu\rho(1,0,-1)+s\rho(0,0,-1)\right\}  \\
 &\qquad +s\eta(0,0)\rho(0,0,0) \\
 &=-c\eta(0,1)\rho(1,0,-1)-s\{\eta(0,1)\rho(0,0,-1)-\eta(0,0)\rho(0,0,0)\} \\
 &\leq 0,
\end{split}
\end{equation}
where the first equality follows from \eqref{rinsetsu.3},
the second equality follows from \eqref{rinsetsu.1} and
the inequality follows from \eqref{FKG.inequality}.

The terms of \eqref{eq:2skappa} not including $\eta(\cdot,\cdot)$, divided by $(d-2)$, are rewritten as
\begin{equation}\label{terms.no.including.tau}
 \begin{split}
 & (\nu-c)\left\{-\rho(0,0,1)+\rho(1,0,0)\right\}\rho(1,0,-1)+(\nu-1)\rho(1,0,0)\rho(0,0,0) \\
 &=-(\nu-c)(\nu-1)\rho(1,0,0)\rho(1,0,-1)-(\nu-c)s\rho(0,0,0)\rho(1,0,-1) \\
 &\qquad + (\nu-1)\rho(1,0,0)\rho(0,0,0) \\
 &\leq -\left\{\frac{(\nu-c)s}{\nu d/(d+2)+s}-(\nu-1)\right\}\rho(1,0,0)\rho(0,0,0) \\
 &= -\frac{(1-c)s-\nu(\nu-1)d/(d+2)}{\nu d/(d+2)+s}\rho(1,0,0)\rho(0,0,0), 
\end{split}
\end{equation}
which is nonpositive for $s\geq \nu(\nu-1)d/\{(1-c)(d+2)\}$,
where the first equality follows from \eqref{rinsetsu.1} and the inequality follows from
\eqref{inequality.000}.

By \eqref{terms.including.tau} and \eqref{terms.no.including.tau},
we have 
\begin{equation*}
 \left(\frac{c}{\nu}-1\right)\|\nabla_u \psi\|^2 + \psi\Delta_u\psi \leq 0,
 \quad \forall u\in\mathbb{R}^d
\end{equation*}
or equivalently
\begin{equation*}
 \Delta_u\left\{E_{Z_1} \left[m^\nu_\HH( t Z_1+u,v)\right]\right\}^{c/\nu}
  \leq 0,\quad\forall u\in\mathbb{R}^d,
\end{equation*}
when $t\leq \{(d+2)(1-c)v/\{d\nu(\nu-1)\}\}^{1/2}$.

\subsection{Proof of Lemma \ref{useful.lemma}}
\label{subsec:useful.lemma}
[Part of \eqref{rinsetsu.0}] \quad Note 
\begin{align}\label{useful}
 \frac{\partial}{\partial \lambda_1}\exp\left(-\frac{z\sum\lambda_i}{\sum\lambda_i+s}\right)
=-\frac{sz}{(\sum\lambda_i+s)^2}\exp\left(-\frac{z\sum\lambda_i}{\sum\lambda_i+s}\right).
\end{align}
Then, by an integration by parts, we have
 \begin{align*}
& sz \int_0^1 \lambda_1^{j_1}\lambda_2^{j_2}(\sum\lambda_i+s)^l \zeta(\lambda)\rd\lambda_1 \\
&=-\lambda_2^{d/2-2+j_2}\prod_{i=3}\lambda_i^{d/2-2}
\int_0^1 \lambda_1^{d/2-2+j_1}(\sum\lambda_i+s)^{l-d/2+2} \\
&\qquad \times\left\{ \frac{\partial}{\partial \lambda_1} 
  \exp\left(-\frac{z\sum\lambda_i}{\sum\lambda_i+s}\right)\right\}\rd\lambda_1 \\
  &=-\lambda_2^{d/2-2+j_2}\prod_{i=3}\lambda_i^{d/2-2}
  \left\{\left[\lambda_1^{d/2-2+j_1}(\sum\lambda_i+s)^{l-d/2+2}\exp\left(-\frac{z\sum\lambda_i}{\sum\lambda_i+s}\right)\right]_0^1  \right. \\
  &\quad \left. -(d/2-2+j_1)\int_0^1
  \lambda_1^{d/2-3+j_1}(\sum\lambda_i+s)^{l-d/2+2}\exp\left(-\frac{z\sum\lambda_i}{\sum\lambda_i+s}\right)\rd \lambda_1 \right. \\
  &\quad \left. -(l-d/2+2)\int_0^1
  \lambda_1^{d/2-2+j_1}(\sum\lambda_i+s)^{l-d/2+1}\exp\left(-\frac{z\sum\lambda_i}{\sum\lambda_i+s}\right)\rd \lambda_1 \right\}.
 \end{align*}
 \eqref{rinsetsu.0} follows from integration with respect to $\lambda_2,\dots,\lambda_\nu$
 in the both hand side of the above equality.
 
\smallskip

[Parts of \eqref{rinsetsu.1}, \eqref{rinsetsu.2} and \eqref{rinsetsu.3}]\quad
The equalities \eqref{rinsetsu.1}, \eqref{rinsetsu.2} and \eqref{rinsetsu.3}
easily follows from symmetry with respect to $\lambda_i$'s.

\smallskip

[Part of \eqref{FKG.inequality}]\quad Note that \eqref{FKG.inequality} is equivalent to
 \begin{align*}
 & \eta(0,0)\rho(0,0,0)-\eta(0,1)\rho(0,0,-1)\\
 & =\{\rho(0,0,0)-\rho(0,0,-1)\}\eta(0,1)-\{\eta(0,1)-\eta(0,0)\}\rho(0,0,0)\\
 &=\int_{\mathcal{D}_{\nu-1}} f_1(\lambda_2,\dots,\lambda_\nu) \prod_{i=2} \rd\lambda_i\int_{\mathcal{D}_{\nu-1}} f_2(\lambda_2,\dots,\lambda_\nu) \prod_{i=2} \rd\lambda_i \\
 &\qquad -\int_{\mathcal{D}_{\nu-1}} f_3(\lambda_2,\dots,\lambda_\nu) \prod_{i=2} \rd\lambda_i
 \int_{\mathcal{D}_{\nu-1}} f_4(\lambda_2,\dots,\lambda_\nu) \prod_{i=2} \rd\lambda_i \\
 & \leq 0,
\end{align*}
where
\begin{align*}
 f_1(\lambda_2,\dots,\lambda_\nu)
 &= \int_0^1 \left(1-\frac{1}{\sum\lambda_i+s}\right)
 \zeta(\lambda_1,\dots,\lambda_\nu) \rd\lambda_1 \\
 f_2(\lambda_2,\dots,\lambda_\nu)
 &= (1+\sum\nolimits_{i=2}\lambda_i+s)
\zeta(1,\lambda_2,\dots,\lambda_\nu) \\
 f_3(\lambda_2,\dots,\lambda_\nu)
 &= (\sum\nolimits_{i=2}\lambda_i+s)
 \zeta(1,\lambda_2,\dots,\lambda_\nu) \\
 f_4(\lambda_2,\dots,\lambda_\nu)
 &= \int_0^1  \zeta(\lambda_1,\dots,\lambda_\nu) \rd\lambda_1.
\end{align*}
 Since both $1-1/\left(\sum\lambda_i+s\right)$ and $ \sum\lambda_i+s$
are increasing in each of its arguments, we have
 \begin{align}
 & \left\{1-1/\left(\sum\lambda_i+s\right)\right\}(1+\sum\nolimits_{i=2}\xi_i+s)\notag\\
 &\leq \left\{1-\frac{1}{(\lambda_1\vee 1) + \sum\nolimits_{i=2}(\lambda_i\vee\xi_i)+s}\right\} 
\left\{(\lambda_1\vee 1) + \sum\nolimits_{i=2}(\lambda_i\vee\xi_i)+s\right\} \notag\\
 & = \sum\nolimits_{i=2}(\lambda_i\vee\xi_i)+s,\label{eq:mtp2.1}
 \end{align}
 where $\vee$ is the maximum operator, i.e.~$\lambda_i\vee\xi_i=\max(\lambda_i,\xi_i)$.
 In the following, $ \wedge$ denotes the minimum operator,
 i.e.~$\lambda_i\wedge\xi_i=\min(\lambda_i,\xi_i)$.
Note that a function $h$: $\mathbb{R}^\nu\to \mathbb{R}$
is said to be multivariate totally positive of order two (MTP2) if it satisfies
\begin{equation*}
 h(x_1,\dots,x_\nu)h(y_1,\dots,y_\nu)
  \leq h(x_1\vee y_1,\dots,x_\nu\vee y_\nu)h(x_1\wedge y_1,\dots,x_\nu\wedge y_\nu)
\end{equation*}
for any $x,y\in \mathbb{R}^\nu$.
 By Lemma \ref{lem:MTP2} below, $\zeta(\lambda_1,\dots,\lambda_\nu)$
is MTP2
as a function of  $\nu$-variate function and hence the inequality
\begin{equation}\label{eq:mtp2.2}
 \begin{split}
& \zeta(\lambda_1,\dots,\lambda_\nu) \zeta(1,\xi_2,\dots,\xi_\nu)\\
& \leq \zeta(\lambda_1\vee 1, \lambda_2\vee \xi_2, \dots, \lambda_\nu\vee \xi_\nu)
\zeta(\lambda_1\wedge 1, \lambda_2\wedge \xi_2, \dots, \lambda_\nu\wedge \xi_\nu) \\
&=\zeta( 1, \lambda_2\vee \xi_2, \dots, \lambda_\nu\vee \xi_\nu)
\zeta(\lambda_1, \lambda_2\wedge \xi_2, \dots, \lambda_\nu\wedge \xi_\nu)
\end{split}
\end{equation}
 follows. 
By \eqref{eq:mtp2.1} and \eqref{eq:mtp2.2}, we have
 \begin{equation}\label{tsukareta.1}
\begin{split}
& f_1(\lambda_2,\dots\lambda_\nu)f_2(\xi_2,\dots\xi_\nu) \\
&\leq \int_0^1\Big[\left\{\sum\nolimits_{i=2}(\lambda_i\vee\xi_i)+s\right\}\zeta( 1, \lambda_2\vee \xi_2, \dots, \lambda_\nu\vee \xi_\nu)  \\
&\qquad\qquad \times\zeta(\lambda_1, \lambda_2\wedge \xi_2, \dots, \lambda_\nu\wedge \xi_\nu) 
 \Big]\rd\lambda_1  \\
 &=f_3(\lambda_2\vee \xi_2, \dots, \lambda_\nu\vee \xi_\nu)
f_4(\lambda_2\wedge \xi_2, \dots, \lambda_\nu\wedge \xi_\nu). 
\end{split}
 \end{equation}
 From Theorem \ref{thm:Karlin-Rinott-1980} below, shown by \cite{Karlin-Rinott-1980},
 the theorem follows.

\smallskip

[Part of \eqref{inequality.000}]\quad  
By Jensen's inequality, we have
\begin{equation}\label{dd+2.1}
\begin{split}
 \frac{\rho(1,0,-1)}{\rho(1,0,0)}
&  =
\int \frac{1}{\lambda_1+\sum_{i=2}^\nu\lambda_i+s}\frac{\lambda_1\zeta(\lambda)}{\rho(1,0,0)} \rd\lambda \\
& \geq \frac{1}{\displaystyle \frac{\rho(2,0,0)}{\rho(1,0,0)}+(\nu-1)\frac{\rho(1,1,0)}{\rho(1,0,0)}+s}.
\end{split} 
\end{equation}
Let $f$ be a probability density given by
\begin{equation*}
f(\lambda_1,\dots,\lambda_\nu)=\frac{d}{2}\left(\frac{d}{2}-1\right)^{\nu-1}
  \lambda_1^{d/2-1}\prod_{i=2}^\nu \lambda_{i}^{d/2-2},
\end{equation*}
which is clearly MTP2. Also let
\begin{equation*}
 g_1(\lambda_1,\dots,\lambda_\nu)
  =\lambda_1,\quad g_2(\lambda_1,\dots,\lambda_\nu)
  =-\frac{\exp\left(sz/\{\sum\lambda_i+s\}\right)}{(\sum\lambda_i+s)^{d/2}},
\end{equation*}
which are both increasing
increasing in each of its arguments.
Hence, by so-called FKG inequality given in Theorem \ref{thm:Karlin-Rinott-1980_FKG} below,
\begin{align*}
 & \int_{\mathcal{D}_{\nu}}
 g_1(\lambda_1,\dots,\lambda_\nu)
 g_2(\lambda_1,\dots,\lambda_\nu)
 f(\lambda_1,\dots,\lambda_\nu) \rd\lambda \\
 &\geq \int_{\mathcal{D}_{\nu}}
 g_1(\lambda_1,\dots,\lambda_\nu)
 f(\lambda_1,\dots,\lambda_\nu) \rd\lambda
 \int_{\mathcal{D}_{\nu}}
 g_2(\lambda_1,\dots,\lambda_\nu)
 f(\lambda_1,\dots,\lambda_\nu) \rd\lambda  
\end{align*}
or equivalently
\begin{align*}
 &  \frac{\int_{\mathcal{D}_{\nu}}
 g_1(\lambda_1,\dots,\lambda_\nu)
 g_2(\lambda_1,\dots,\lambda_\nu)
 f(\lambda_1,\dots,\lambda_\nu) \rd\lambda}
 {\int_{\mathcal{D}_{\nu}}
 g_2(\lambda_1,\dots,\lambda_\nu)
 f(\lambda_1,\dots,\lambda_\nu) \rd\lambda} \\
 &\leq \int_{\mathcal{D}_{\nu}}
 g_1(\lambda_1,\dots,\lambda_\nu)
 f(\lambda_1,\dots,\lambda_\nu) \rd\lambda,
\end{align*}
since $ g_2< 0$.
Since $\rho(2,0,0)/\rho(1,0,0)$ is expressed as
\begin{equation*}
 \frac{\rho(2,0,0)}{\rho(1,0,0)}
  =\frac{\int_{\mathcal{D}_{\nu}}
  g_1(\lambda_1,\dots,\lambda_\nu)
  g_2(\lambda_1,\dots,\lambda_\nu)
  f(\lambda_1,\dots,\lambda_\nu) \rd\lambda}
  {\int_{\mathcal{D}_{\nu}}
  g_2(\lambda_1,\dots,\lambda_\nu)
  f(\lambda_1,\dots,\lambda_\nu) \rd\lambda},
\end{equation*}
we have
\begin{equation}\label{dd+2.2}
 \frac{\rho(2,0,0)}{\rho(1,0,0)}\leq \frac{d}{d+2}.
\end{equation}
Similarly we have
\begin{equation}\label{dd+2.3}
 \frac{\rho(1,1,0)}{\rho(1,0,0)}\leq \frac{d-2}{d} \leq \frac{d}{d+2}.
\end{equation}
Hence, by \eqref{dd+2.1}, \eqref{dd+2.2} and \eqref{dd+2.3}, we have
\begin{equation*}
 \frac{\rho(1,0,-1)}{\rho(1,0,0)} \geq \frac{1}{\nu d/(d+2)+s}.
\end{equation*}

 \begin{lemma}\label{lem:MTP2}
 Let
\begin{align*}
 \zeta(\lambda_1,\dots,\lambda_\nu)
 =\frac{\prod \lambda_i^{d/2-2}}{ (\sum\lambda_i+s)^{d/2}}
 \exp\left(-\frac{\sum\lambda_i}{\sum\lambda_i+s}z\right).
\end{align*}
 Then $ \zeta(\lambda_1,\dots,\lambda_\nu)$ is MTP2.
 \end{lemma}
\begin{proof}
Note
\begin{align*}
 \exp\left(-\frac{\sum\lambda_i}{\sum\lambda_i+s}z\right)
 =\exp(-z)\exp\left(\frac{sz}{\sum\lambda_i+s}\right).
\end{align*}
From the form of $\zeta$, we have only to check
\begin{align*}
 (\sum\lambda_i+s)(\sum\xi_i+s)
 \geq  (\sum\lambda_i\vee\xi_i+s)(\sum\lambda_i\wedge\xi_i+s)
\end{align*}
or equivalently
\begin{align*}
 (\sum\lambda_i)(\sum\xi_i)
 \geq  (\sum\lambda_i\vee\xi_i)(\sum\lambda_i\wedge\xi_i).
\end{align*}
We have
\begin{align*}
 & (\sum\lambda_i)(\sum\xi_i)- (\sum\lambda_i\vee\xi_i)(\sum\lambda_i\wedge\xi_i) \\
 &=\sum_{i\neq j}\left\{\lambda_i\xi_j+\lambda_j\xi_i
 -(\lambda_i\vee\xi_i)(\lambda_j\wedge\xi_j)-(\lambda_j\vee\xi_j)(\lambda_i\wedge\xi_i)\right\} .
\end{align*}
Without the loss of generality, assume $\lambda_i\geq\xi_i$. Then we have
\begin{align*}
 & \lambda_i\xi_j+\lambda_j\xi_i
 -(\lambda_i\vee\xi_i)(\lambda_j\wedge\xi_j)-(\lambda_j\vee\xi_j)(\lambda_i\wedge\xi_i) \\
 &=\lambda_i\xi_j+\lambda_j\xi_i -\lambda_i(\lambda_j\wedge\xi_j)-(\lambda_j\vee\xi_j)\xi_i \\
 &=\lambda_i\{\xi_j-(\lambda_j\wedge\xi_j)\} -\xi_i\{(\lambda_j\vee\xi_j)-\lambda_j\} \\
 &=(\lambda_i-\xi_i)\{\xi_j-(\lambda_j\wedge\xi_j)\} \\
 &\geq 0,
\end{align*}
which completes the proof.
\end{proof}

\begin{thm}[Theorem 2.1 of \cite{Karlin-Rinott-1980}]\label{thm:Karlin-Rinott-1980}
 Let $f_1,f_2,f_3$ and $f_4$ be nonnegative functions satisfying for all $x,y\in\mathbb{R}^\nu$
\begin{align*}
 f_1(x)f_2(y)
 \leq f_3(x\vee y)f_4(x\wedge y).
\end{align*}
 Then
 \begin{align*}
  \int f_1(x)\rd x  \int f_2(x)\rd x 
  \leq \int f_3(x)\rd x  \int f_4(x)\rd x .
 \end{align*}
\end{thm}

\begin{thm}[FKG Inequality, e.g.~Theorem 2.3 of \cite{Karlin-Rinott-1980}]\label{thm:Karlin-Rinott-1980_FKG}
Let $f(x)$ for $x\in\mathbb{R}^\nu$ be a probability density satisfying MTP2.
Then for any pair of increasing functions $g_1(x)$ and $g_2(x)$, we have
 \begin{align*}
  \int g_1(x)g_2(x)f(x)\rd x 
  \geq \int g_1(x)f(x)\rd x \int g_2(x)f(x)\rd x .
 \end{align*}
\end{thm}

\end{document}